\documentclass[a4paper;12pt]{amsart}

\usepackage{latexsym,amssymb,amsmath,amsopn,graphics,xy,epsfig,picture,epic}

\usepackage{color}

\def\aa{{\mathcal A}}

\def\cc{{\mathcal C}}

\def\mm{{\mathcal M}}

\def\ss{{\mathcal S}}

\def\ffi{\varphi}
\def\eps{\varepsilon}
\def\dst{\displaystyle}

\renewcommand{\Im}{\mathrm{Im}\,}

\DeclareMathOperator{\supp}{supp}

\DeclareMathOperator{\cotan}{cotan}

\def\N{{\mathbb{N}}}

\def\Q{{\mathbb{Q}}}
\def\R{{\mathbb{R}}}
\def\S{{\mathbb{S}}}
\def\T{{\mathbb{T}}}

\def\Z{{\mathbb{Z}}}

\newcommand{\norm}[1]{{\left\|{#1}\right\|}}
\newcommand{\ent}[1]{{\left[{#1}\right]}}

\newcommand{\scal}[1]{{\left\langle{#1}\right\rangle}}

\newenvironment{notation}[1][]{\vskip1pt\noindent\rm\textit{Notation}\,:\ }{\rm\vskip1pt}
\newenvironment{definition}[1][]{\vskip3pt\noindent\sl\textbf{Definition.}\ }{\rm\vskip3pt}
\newenvironment{remark}[1][]{\vskip3pt\noindent\textbf{Remark.}\ }{\rm\vskip3pt}
\newenvironment{example}[1][]{\vskip3pt\noindent\textbf{Example.}\ }{\rm\vskip3pt}

\newtheorem{lemma}{Lemma}[section]
\newtheorem{proposition}[lemma]{Proposition}
\newtheorem{theorem}[lemma]{Theorem}
\newtheorem{corollary}[lemma]{Corollary}

\date{\today}

\begin{document}
\title{A dynamical system approach to Heisenberg Uniqueness Pairs}

\author{Philippe Jaming \& Karim Kellay}

\address{Univ. Bordeaux, IMB, UMR 5251, F-33400 Talence, France.
CNRS, IMB, UMR 5251, F-33400 Talence, France.}
\email{Philippe.Jaming@u-bordeaux1.fr, Karim.Kellay@math.u-bordeaux1.fr}

\begin{abstract}
Let $\Lambda$ be a set of lines in $\mathbb{R}^2$ that intersect at the origin. For 
$\Gamma\subset\mathbb{R}^2$ a smooth curve, we denote by $\mathcal{A}\mathcal{C}(\Gamma)$ the 
subset of finite measures on $\Gamma$ that are absolutely continuous with respect to arc length 
on $\Gamma$. For $\mu\in \mathcal{A}\mathcal{C}(\Gamma)$, $\widehat{\mu}$ denotes the Fourier transform of $\mu$.
Following Hedenmalm and Montes-Rodr\'iguez, we will say that $(\Gamma,\Lambda)$ is a Heisenberg 
Uniqueness Pair if $\mu\in\mathcal{A}\mathcal{C}(\Gamma)$ is such that $\widehat{\mu}=0$
on $\Lambda$, then $\mu=0$. The aim of this paper is to provide new tools to establish this 
property. To do so, we will reformulate the fact that $\widehat{\mu}$ vanishes on $\Lambda$ in 
terms of an invariance property of $\mu$ induced by $\Lambda$. This leads us to a dynamical system 
on $\Gamma$ generated by $\Lambda$. In many cases, the investigation of this dynamical system allows us to 
establish that $(\Gamma,\Lambda)$ is a Heisenberg Uniqueness Pair. This way we both unify proofs 
of known cases (circle, parabola, hyperbola) and obtain many new examples. This method also allows 
to have a better geometric intuition on why $(\Gamma,\Lambda)$ is a Heisenberg Uniqueness Pair.
As a side result, we also give the first instance of a positive result in the classical
Cram\'er-Wold theorem where finitely many projections suffice to characterize a measure (under strong support
constraints).
\end{abstract}

\subjclass{42A68;42C20}

\keywords{Uncertainty principles;annihilating pairs;Heisenberg pairs;Cram\'er-Wold Theorem}

\maketitle


\section{Introduction}

The aim of this paper is to contribute to the study of Fourier uniqueness sets of measures supported on planar curves.
More precisely, in the terminology introduced in \cite{HMR}, we will provide new tools for proving that a 
piecewise smooth curve $\Gamma$ and a set $\Lambda$ of lines through the origin form a \emph{Heisenberg Uniqueness Pairs} (HUP).

This concept of HUP is an extension of 
the notion of annihilating pairs for the Fourier transform on $L^2(\R)$ to the setting of measures
{\it see e.g.} \cite{AB,Be}, Havin and J\"oricke's book \cite{HJ} or the survey \cite{FS}.
Its original motivations comes from sets of uniqueness of PDEs (in particular for the Klein-Gordon equation).
We will show that the problem can be reformulated in terms of a dynamical system on $\Gamma$. This will allow us to
find new proofs for many existing results as well as to find many new cases that seemed out of reach with
the methods used so far.

Let us now be more precise. If $\mu$ denotes a finite complex-valued Borel measure in the plane $\R^2$. The Fourier transform 
of $\mu$ is defined by
$$
\widehat{\mu}(x,y)=\int_{\R^2} e^{-i(xs+yt)}\,\mbox{d}\mu(s,t).
$$

For $\Gamma\subset\R^2$ that is the finite union of smooth curves that are disjoint (except possibly for 
the endpoints),
denote by $\mm(\Gamma)$ the set of finite complex-valued Borel measures supported in $\Gamma$. Moreover, we denote 
by $\aa\cc(\Gamma)$ the subset of $\mm(\Gamma)$ that consists of measures that are absolutely continuous with respect 
to arc length on $\Gamma$.

\begin{definition} 
Let $\Lambda\subset\R^2$ and $\Gamma$ a finite union of smooth disjoint curves. Then
$(\Gamma,\Lambda)$ is a \emph{Heisenberg Uniqueness Pair} if $\mu\in\aa\cc(\Gamma)$ and
$\widehat{\mu}\Big|_{\Lambda}=0$ implies $\mu=0$.
\end{definition}

Clearly, some of the invariance properties of the Fourier transform transfer to HUPs, namely:

\begin{itemize}
\item[{\em [Inv 1]}] Fix $(s_0,t_0),(x_0,y_0)\in\R^2$. Then $\bigl(\Gamma,\Lambda\bigr)$ is a HUP if and only if
$\bigl(\Gamma-(s_0,t_0),\Lambda-(x_0,y_0)\bigr)$ is a HUP.

\item[{\em [Inv 2]}] Fix $T$ a linear invertible transformation $\R^2\to\R^2$ and denote by $T^*$ its adjoint.
Then $\bigl(\Gamma,\Lambda\bigr)$ is a HUP if and only if $\bigl(T^{-1}(\Gamma),T^*(\Lambda)\bigr)$
is a HUP.
\end{itemize}

This notion was introduced by Hedenmalm and Montes-Rodr\'iguez \cite{HMR} who considered the
case where $\Gamma$ is a hyperbola $\{(x,y\in\R^2\,: xy=1\}$ and $\Lambda=\alpha\Z\times\{0\}\cup\{0\}\times\beta\Z$
is the lattice cross {\it i.e.} a discrete set included in two lines. The case of $\Gamma$ an ellipse and $\Lambda$
two lines was soon after settled independently by Sj\"olin 
\cite{Sjo1} and Lev \cite{Le}. Finally Sj\"olin
\cite{Sjo} considered the case where $\Gamma$ is a parabola, thus completing the study of quadratic curves.

Our aim here is to give more geometric proofs of the results of Sj\"olin and Lev that allow us to extend their
results to the case where $\Gamma$ is a rather general curve and $\Lambda$ is a union of two intersecting lines.
According to the invariance properties we can assume that the lines intersect at the origin
and write $\ell_\theta=\{(t\cos\theta,t\sin\theta),\ t\in\R\}$ for $\theta\in[0,\pi)$.

Our starting point was Sj\"olin's proof that parabolas and two well chosen lines form an HUP. In particular, Sj\"olin
used a simple change of variable that directly reformulates as Lemma \ref{lem:fund}-Corollary \ref{cor:2} in our case.
These results show that, for $\Lambda=\ell_{\theta_1}\cup\ell_{\theta_2}$ a set of two lines through the origin,
if $\mu\in\aa\cc(\Gamma)$ and $\widehat{\mu}\Big|_{\Lambda}=0$ then there is a mapping
$\Phi\,:\Gamma\to\Gamma$ that leaves $\mu$ invariant. Moreover this mapping has a simple geometric interpretation.
We will then be able to deduce from the properties of the dynamical system generated by $\Phi$ (existence of a wandering set,
existence of attractive points and ergodicity) that $(\Gamma,\Lambda)$ is a Heseinberg Uniqueness Pair.
Note that dynamical systems already play a crucial role in \cite{HMR,CMHMR}.

Let us here summarize our main results:

\medskip

\noindent{\bf Main Theorem.} {\sl Let $\Gamma$ be any of the following curves:
\begin{enumerate}
\renewcommand{\theenumi}{\roman{enumi}}
\item\label{mt2} the graph of $\psi(t)=|t|^\alpha$, $t\in\R$, $\alpha>0$;

\item\label{mt3} a hyperbola;

\item\label{mt4} a polygon;

\item\label{mt5} an ellipse.
\end{enumerate}

Then there exists a set $E\subset(-\pi/2,\pi/2)\times(-\pi/2,\pi/2)$ of positive measure such that,
if $(\theta_1,\theta_2)\in E$, $(\Gamma,\ell_{\theta_1}\cup\ell_{\theta_2})$ is a Heisenberg Uniqueness Pair.}

\medskip

The actual results are both more general and more precise, we refer to Theorem \ref{cor:thmain} for \ref{mt2}),
and to Proposition \ref{prop:poly}
for \ref{mt4}). To prove those results we show that $\Phi$ has many wandering sets.
We prove \ref{mt3}) in Theorem \ref{th:hyp} by first transferring the problem to the circle
(using a simple transform from projective geometry) in order to prove that here too
$\Phi$ has many wandering sets. Finally, the case \ref{mt5}) is proved in Theorem \ref{th:circle}
using ergodic theory. In this case, the map $\Phi$ is an irrational rotation. Our technique shows that
the same result holds if $\Gamma$ is any smooth convex closed curve such that the map $\Phi$
has irrational rotation number. However we are also able to construct an example of a 
smooth convex closed curve and a set of two lines that form a HUP and such that the map $\Phi$ has rational rotation number.

\medskip

Let us now explain how our results apply to PDEs.

Let $p$ be a polynomial of two variables and let $\Gamma=\{(s,t)\in\R^2\,: p(s,t)=0\}$. Then
\begin{equation}
\label{eq:pdfF}
p(i\partial_x,i\partial_y)\widehat{\mu}(x,y)
=\int_{\R^2}e^{-i(xs+yt)}p(s,t)\,\mbox{d}\mu(s,t).
\end{equation}
Therefore, if $\mu\in\aa\cc(\Gamma)$ then $F=\widehat{\mu}$ solves the PDE
\begin{equation}
\label{eq:PDE}
p(i\partial_x,i\partial_y)F=0.
\end{equation}

Now $(\Gamma,\ell_{\theta_1}\cup\ell_{\theta_2})$ is a Heisenberg Uniqueness pair
if and only if for every solution $F$ of \eqref{eq:PDE} such that
$F=\widehat{\mu}$ with $\mu\in\aa\cc(\Gamma)$, $F(x,x\cotan \theta_1)=F(x,x\cotan \theta_2)=0$
for every $x$ implies $F=0$.

We can then reformulate our results in terms of solutions of certain PDEs (and more generally
for certain pseudo-differential equations). The following theorem is then a reformulation of the main theorem.

\medskip

\noindent{\bf Theorem.} {\sl Let $\theta_1\not=\theta_2\in(0,\pi)$, $a_j=\cotan \theta_j$ and $\alpha>0$.
Assume that $F\in\cc^2(\R^2)$ satisfies one of the following equations:
\begin{enumerate}
\renewcommand{\theenumi}{\roman{enumi}}
\item \emph{Shr\"odinger Equation} 
$$
i\partial_x F\pm |\Delta_y|^{\alpha/2}F=0.
$$
In this case, take $\Gamma=\{(t,|t|^\alpha),t\in\R\}$;

\item \emph{Helmholtz equation} 
$$
\partial_x^2 F+\partial_y^2F=-\alpha^2F.
$$
In this case, take $\Gamma=\{(x,y)\in\R^2\,:x^2+y^2=\alpha^2\}$ and further assume that $\dst\frac{\theta_1-\theta_2}{\pi}\notin\Q$;

\item \emph{Klein-Gordon equation} 
$$
\partial_x^2 F-\partial_y^2F=\alpha^2F.
$$
In this case, take $\Gamma=\{(x,y)\in\R^2\,:x^2-y^2=\alpha^2\}$ and further assume that $|\theta_1-\theta_2|\not=\pi/2$.
\end{enumerate}

If  $F=\widehat{\mu}$ with $\mu\in\aa\cc(\Gamma)$ and 
\begin{equation}
\label{eq:pde}
F(x,a_1x)=F(x,a_2x)=0\qquad\mbox{ for all }x\in\R,
\end{equation}
then $F=0$.}

\medskip

Moreover, the result is true if, for the Shr\"odinger Equation, \eqref{eq:pde} only holds for $x$ in a set of positive measure,
while for the Helmholtz Equation, \eqref{eq:pde} needs only to hold for $x$ in a discrete set.
\medskip

One would of course like to relax the condition $F=\widehat{\mu}$ with $\mu\in\aa\cc(\Gamma)$ to $F=\widehat{\mu}$
with $\mu$ a bounded measure on $\R^2$ (which would then necessarily be supported in $\Gamma$). It would be natural
to say that $(\Gamma,\ell_{\theta_1}\cup\ell_{\theta_2})$ is a \emph{strong} Heisenberg Uniqueness Pair
in that case.

\medskip

Another application of our results is to the classical Cram\'er-Wold Theorem on the
characterization of a probability measure from its projections. In order to state
those results, we need some further notation. Let $\mm(\R^2)$ be the set of finite signed measures on $\R^2$. The Radon transform
can be defined on $\mm(\R^2)$ in various equivalent ways. Roughly speaking,
for $\theta\in\S^1$, the Radon transform of $\mu\in\mm(\R^2)$ in direction $\theta$
is the marginal probability measure of $\mu$ in direction $\theta$.
To be more rigorous, we first need to define \emph{dual Radon transform}:
for a bounded $g\in \cc(\S^1\times\R)$ and $x\in\R^2$ 
$$
R^*[g](x)=\int_{\S^1}g(\theta,\scal{x,\theta})\,\mbox{d}\theta.
$$

\begin{definition}
For $\mu\in\mm(\R^2)$, the Radon transform of $\mu$ is the measure
$\nu=R[\mu]\in\mm(\S^1\times\R)$ defined by
\begin{equation}
\label{eq:radon2}
\int_{\S^1\times\R}g(\theta,s)\,\mbox{d}\nu(\theta,s)
=\int_{\R^2}R^*[g](x)\,\mbox{d}\mu(x)\qquad\mbox{for every }g\in \cc_c(\S^1\times\R).
\end{equation}
\end{definition}

It seems difficult to trace back the first occurrence of the Radon Transform for measures.
The properties we need can be found in \cite{HHK,HQ,BL}. In particular, the following properties
have been established: let $\mu\in\mm(\R^2)$ $\nu=R[\mu]\in\mm(\S^1\times\R)$

\begin{enumerate}
\renewcommand{\theenumi}{\roman{enumi}}
\item Indeed, $R[\mu](\theta,\cdot)$ is the push-forward $\pi_\theta*\mu$
where $\pi_\theta$ is the projection $\R^2\to\R$ given by $\pi_\theta x=\scal{x,\theta}$, that is
for every Borel set $E\subset\R$, $R[\mu](\theta,E)=\mu\bigl(\pi_\theta^{-1}(E)\bigr)$.

\item let $\widehat{\nu}_\xi$ be the partial Fourier transform in the ``$s$''-variable, that is
for $\xi\in\R$, $\widehat{\nu}_\xi$ is the measure defined on $\S^1$ by
$$
\int_{S^1}G(\theta)\,\mbox{d}\widehat{\nu}_\xi(\theta)=
\int_{\S^1}\int_\R G(\theta)e^{-is\xi}\,\mbox{d}\nu(\theta,s)
$$
for every $G\in\cc(\S^1)$.

\begin{theorem}[Fourier-Slice Theorem]
For $\mu\in\mm(\R^2)$, $\widehat{\nu}_\xi=\widehat{\mu}(\xi\theta)\,\mbox{d}\theta$.
\end{theorem}

\item For $\theta\in\S^1$ and $s\in\R$, let $H_{\theta,s}=\{x\in\R^2\,\scal{x,\theta}<s\}$. Then
for every $\ffi\in\cc^\infty_c(\R)$,
$$
\scal{R[\mu](\theta,\cdot),\ffi}=-\int_\R\mu(H_{\theta,s})\ffi'(s)\,\mbox{d}s.
$$
In other words, $R[\mu](\theta,\cdot)$ is the derivative in the sense of distributions of the 
function (of bounded variation) $s\to \mu(H_{\theta,s})$.
\end{enumerate}

Now, according to the Fourier-Slice Theorem, $(\Gamma,\ell_{\theta_1}\cup\ell_{\theta_2})$
is a Heisenberg Uniqueness Pair if and only if 
the only $\mu\in\aa\cc(\Gamma)$ such that $R[\mu](\theta_1,\cdot)=R[\mu](\theta_2,\cdot)=0$
is $\mu=0$.

This immediately leads to the following refinements of the celebrated Cram\'er-Wold Theorem
\cite{CW}:

\medskip

\noindent{\bf Restricted Cram\'er-Wold Theorem} {\sl Let $\Gamma$ be any of the curves mentioned in the main theorem
and let $E\subset(-\pi/2,\pi/2)\times(-\pi/2,\pi/2)$ be the corresponding set of positive measure.
Let $(\theta_1,\theta_2)\in E$. Let $\mu\in\aa\cc(\Gamma)$, if 
$R[\mu](\theta_1,\cdot)=R[\mu](\theta_2,\cdot)=0$ then $\mu=0$.}

\medskip
This seems to be the first instance of a positive result in the Cram\'er-Wold theorem
for finitely many angles. For infinitely many angles, we refer to \cite{BMR}.

It seems reasonable to say that in those cases $(\Gamma,\{\theta_1,\theta_2\})$ form a
\emph{Cram\'er-Wold Pair}. 
This leads to an other notion of Strong Heisenberg Uniqueness Pair, that we
call a \emph{Strong Cram\'er-Wold Pair}:
if $\Gamma\subset\R^2$ is a curve and $\Lambda\subset\S^1$ then $(\Gamma,\Lambda)$ 
is a Strong Cram\'er-Wold Pair if for $\mu_n,\mu\in\aa\cc(\Gamma)$,
$\mu_n\to\mu$ weakly if and only if $R[\mu_n](\theta,\cdot)\to R[\mu](\theta,\cdot)$
weakly for every $\theta\in\Lambda$. Here we follow the probabilist's tradition
to say that $\mu_n\to\mu$ weakly if $\int\ffi\,\mbox{d}\mu_n\to\int\ffi\,\mbox{d}\mu$
for every bounded continuous function $\ffi$, that is, if $\mu_n\to\mu$ in the weak-$*$
topology. A part in the case of the full Cram\'er-Wold Theorem ({\it see} \cite{HQ})
not much seems to be known in this direction.  

\medskip                                                 

The remaining of the paper is organized as follows. The following section is devoted to the 
technical lemmas we will need. In particular, Section \ref{sec:general} contains the three technical
lemmas linking Heisenberg Uniqueness Pairs and properties of the dynamical system generated by $\Phi$.
Section \ref{sec:wandering} is then devoted to cases where the dynamical system has many wandering sets,
in particular establishing \ref{mt2}) to \ref{mt4}) of the Main theorem in four consecutive subsections.
The last section is devoted to closed curves when the map $\Phi$ has a rotation number.

\section{Technical Lemmas}

\subsection{Notation and key lemma}

Throughout this paper, $I$ will be a finite union of disjoint intervals and 
$\Gamma=\{\gamma(s),s\in I\}$ will be a curve in the
plane parametrized by a function $\gamma\,:I\to\R^2$ that is assumed to be piecewise $\cc^k$-smooth ($k\geq1$) and one-to-one 
(except possibly for the end points of $I$). 

For $\theta\in \S^1$ the unit circle of $\R^2$ denote by $\theta^\perp$ the vector in $\S^1$ directly orthogonal to 
$\theta$. We will use the common abuse of notation by identifying $\theta$ with its the angle with the horizontal
axes, $\theta=(\cos\theta,\sin\theta)$ so that $\theta^\perp=(-\sin\theta,\cos\theta)$. Let $\ell_\theta=\{t\theta,t\in\R\}$ be the line spanned by 
$\theta$ and define $\pi_\theta x=\scal{x,\theta}$ so that $x\to\pi_\theta(x)\theta$ is the
orthogonal projection of $x$ on $\ell_\theta$.

Given $\mu\in\aa\cc(\Gamma)$ {\it i.e.} a measure that is absolutely continuous with respect to arc length on $\Gamma$ 
we write $\mu(s)=g_\mu(s)\norm{\gamma'(s)}\,\mbox{d}s=f_\mu(s)\,\mbox{d}s$, with $f_\mu\in L^1(I)$.

We are now in position to prove the following simple but key lemma:

\begin{lemma}\label{lem:fund}
Assume that $\Gamma$, $\theta$ are such that there exists a finite partition of $I=\bigcup_{k=1}^NI_k$
of intervals that are disjoint (up to the endpoints) such that $s\to \pi_\theta\gamma(s)$ is one-to-one on each $I_k$.

Let $\mu\in\aa\cc(\Gamma)$.
Then $\hat\mu(\xi)=0$ for $\xi\in\ell_\theta$ if and only if, for almost every $\zeta\in\R$
\begin{equation}
\label{eq:fund}
\sum_{s\in\pi_\theta\gamma^{-1}(\zeta)}\frac{f_\mu(s)}{\pi_\theta \gamma'(s)}=0.
\end{equation}
\end{lemma}

\begin{proof} Note that, for $\theta$ fixed,  $(\pi_\theta\gamma)':=\partial_s\pi_\theta\gamma=\pi_\theta\gamma'$. Then
\begin{eqnarray*}
\hat\mu(t\theta)&=&\int_I f_\mu(s)e^{-i t\scal{\gamma(s),\theta}}\,\mbox{d}s\\
&=&\sum_{k=1}^N\int_{I_k} f_\mu(s)e^{-i t\pi_\theta\gamma(s)}\,\mbox{d}s\\
&=&\sum_{k=1}^N\int_{\pi_\theta\gamma^{-1}(I_k)}f_\mu\bigl(\pi_\theta\gamma^{-1}(\zeta)\bigr)
e^{-i t\zeta}
\frac{\mbox{d}\zeta}{\pi_\theta\gamma'\bigl(\pi_\theta\gamma^{-1}(\zeta)\bigr)}
\end{eqnarray*}
with the change of variable $s=\pi_\theta\gamma^{-1}(\zeta)$ on each $I_k$.
It follows that
\begin{eqnarray*}
\hat\mu(t\theta)&=&\int_\R 
\sum_{k=1}^N\mbox{1}_{\pi_\theta\gamma^{-1}(I_k)}(\zeta)\frac{f_\mu\bigl(\pi_\theta\gamma^{-1}(\zeta)\bigr)}
{\pi_\theta\gamma'\bigl(\pi_\theta\gamma^{-1}(\zeta)\bigr)}e^{-i t\zeta}\,\mbox{d}\zeta\\
&=&\int_\R\sum_{s\in\pi_\theta\gamma^{-1}(\zeta)}\frac{f_\mu(s)}{\pi_\theta \gamma'(s)}
e^{-i t\zeta}\,\mbox{d}\zeta.
\end{eqnarray*}
This is now an ordinary Fourier transform so that $\hat\mu(t\theta)=0$ for every $t$ if and only if
\eqref{eq:fund} is satisfied.
\end{proof}

\begin{remark}\ \\
--- If $\gamma$ is contained in a half place $\{\scal{x,\theta^\perp}\geq\alpha\}$
or $\{\scal{x,\theta^\perp}\leq\alpha\}$, then it is enough to assume that $\widehat{\mu}(t\theta)=0$
for $t\in E$ a set of finite positive measure for \eqref{eq:fund} to hold.

\smallskip

This follows immediately from the previous proof and the well known fact ({\it see e.g.} \cite[Page 36]{HJ}) that
if $f\in L^1(\R)$ is such that $\supp f\subset[0,+\infty)$ and if
$$
\int_\R\frac{\log|\widehat{f}(\xi)|}{1+|\xi|^2}\,\mbox{d}\xi=-\infty
$$
(in particular if $\widehat{f}$ is compactly supported) then $f=0$.

\noindent--- Further, if $\Gamma$ is contained in a strip $\{-\alpha\leq\scal{x,\theta^\perp}\leq\alpha\}$,
using the sampling theorem, we may further restrict $E$ to be a discrete set of density $\geq\frac{\alpha}{2\pi}$.
\end{remark}

From now on, we will restrict our attention to curves for which $(\pi_\theta\gamma)^{-1}(\zeta)$ contains at most two points.
More precisely, the following is a direct reformulation of Lemma \ref{lem:fund}:

\begin{corollary}\label{cor:2}
Let $\gamma\,: I\to\R^2$ be a piecewise smooth function and $\theta\in[0,2\pi)$. Assume that we may split
$I=I_0\cup I_-\cup I_+$ in such a way that

\begin{enumerate}
\renewcommand{\theenumi}{\roman{enumi}}
\item $\pi_\theta\gamma$ is one-to-one on each interval $I_0$, $I_+$, $I_-$.

\item let $\sigma\in I$ and $\zeta=\pi_\theta\gamma(\sigma)$ and consider the equation $\pi_\theta\gamma(s)=\zeta$. Then 

-- if $\sigma\in I_0$ this equation has as unique solution $s=\sigma$;

-- if $\sigma\in I_-$ (resp. $I_+$) this equation has two solutions $\sigma_\pm$ with $\sigma_-=\sigma\in I_-$ and $\sigma_+\in I_+$
(resp. $\sigma_+=\sigma\in I_+$ and $\sigma_-\in I_-$). In this case, we denote $\pi_\theta\gamma^{-1}_\pm(\zeta)=\sigma_\pm$
\end{enumerate}

Let $\mu\in\aa\cc(\Gamma)$.
Then $\hat\mu(\xi)=0$ for $\xi\in\ell_\theta$ if and only if,
\begin{enumerate}
\renewcommand{\theenumi}{\roman{enumi}}
\item $f_\mu=0$ on $I_0$

\item for every $s_-\in I_-$, $s_+\in I_+$, with $\pi_\theta\gamma(s_-)=\pi_\theta\gamma(s_+)$
that is $s_+= \pi_\theta\gamma^{-1}_+\bigl(\pi_\theta\gamma(s_-)\bigr)$ and
$s_-= \pi_\theta\gamma^{-1}_-\bigl(\pi_\theta\gamma(s_+)\bigr)$,
\begin{equation}
\label{eq:symf}
\frac{f_\mu(s_+)}{\pi_\theta\gamma'(s_+)}
=-\frac{f_\mu(s_-)}{\pi_\theta\gamma'(s_-)}.
\end{equation}
\end{enumerate}
Moreover, if $\alpha_+,\beta_+\in I_+$ and $\alpha_-,\beta_-\in I_-$ are such that
$\pi_\theta\gamma(\alpha_+)=\pi_\theta\gamma(\alpha_-)$, $\pi_\theta\gamma(\beta_+)=\pi_\theta\gamma(\beta_-)$
then
\begin{equation}
\label{eq:symintsign}
\int_{\alpha_-}^{\beta_-} f_\mu(s_-)\,\mbox{d}s_-=-
\int_{\alpha_+}^{\beta_+} f_\mu(s_+)\,\mbox{d}s_+.
\end{equation}
and
\begin{equation}
\label{eq:symint}
\int_{[\alpha_-,\beta_-]} |f_\mu(s_-)|\,\mbox{d}s_-=
\int_{[\alpha_+,\beta_+]} |f_\mu(s_+)|\,\mbox{d}s_+.
\end{equation}
\end{corollary}

\begin{figure}[!ht]
\begin{center}
\setlength{\unitlength}{0.5cm}
\begin{picture}(12,8)
\qbezier[400](1,7)(4,-5)(12,9)
\dottedline{0.2}(1,7)(13,7)
\put(2.5,6.3){${\scriptstyle s\in I_-}$}
\put(7,6.3){$\scriptstyle s\in I_+$}
\put(11.5,7.5){$\scriptstyle s\in I_0$}
\put(0,1.4){\vector(1,0){12}}
\put(4.9,1){\vector(0,1){8}}
\put(12,1.1){$\theta^\perp$}
\put(5.2,8.5){$\theta$}
\dottedline{0.2}(1.7,4)(9,4)
\put(2.7,4.1){${\scriptstyle \pi_\theta\gamma(s_-)=\pi_\theta\gamma(s_+)}$}
\put(0,3.8){$\gamma(s_-)$}
\put(0,0.2){${\scriptstyle s_-=\pi_\theta\gamma^{-1}_-\bigl(\pi_\theta\gamma(s)\bigr)}$}
\put(8.1,0.2){${\scriptstyle s_+=\pi_\theta\gamma^{-1}_+\bigl(\pi_\theta\gamma(s)\bigr)}$}
\put(8.8,3.8){$\gamma(s_+)$}
\end{picture}
\caption{The notation of Corollary \ref{cor:2}}
\label{fig:cor}
\end{center}
\end{figure}
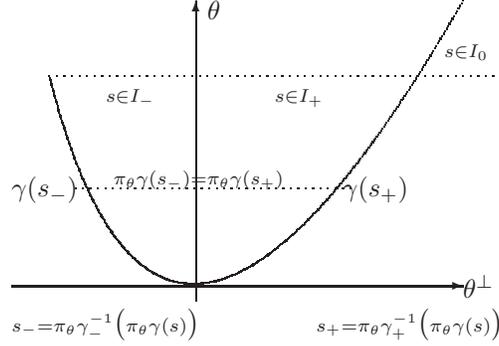

Note that \eqref{eq:symintsign}-\eqref{eq:symint} follows directly from \eqref{eq:symf} if we change variable 
$s_+=\pi_\theta\gamma^{-1}_-\bigl(\pi_\theta\gamma(s_-)\bigr)$
in the second integral.

\begin{notation}
If $\gamma, \theta$ satisfy the hypothesis of Corollary \ref{cor:2} we may define the map 
$$
\Phi_\theta\,:\begin{matrix}I_-\cup I_+\to I_-\cup I_+\\
\Phi_\theta(s_\pm)=s_\mp\end{matrix}.
$$
This map has a nice geometric interpretation:
Consider a point $\gamma(s)$ with $s\in I$ and draw a line orthogonal to $\theta$ starting at $\gamma(s)$.
This line will intersect $\Gamma$ again in $\gamma\bigl(\Phi_\theta(s)\bigr)$.
\end{notation}

Let us now give a first application:

\begin{proposition}
\label{prop:1line}
Let $\psi\,:\R\to\R$ be a continuous piecewise $\cc^1$  function, such that $\psi$ is concave on $\R^-$ and convex on $\R^+$
and that $\psi$ has a left and a right derivative in $0$, $\psi'(0^\pm)$. Let 
$$
\theta_0=\begin{cases}\dst\frac{\pi}{2}&\mbox{if }\psi'(0^-)>0\mbox{ and }\psi'(0^+)>0\\
-\mathrm{arg\,cotan}\min\bigl(\psi'(0^-),\psi'(0^+)\bigr)&\mbox{otherwise}
\end{cases}.
$$
Let $\Gamma=\bigl\{\bigl(s,\psi(s)\bigr),s\in\R\bigr\}$ be the graph of $\psi$. Then there is a $\theta_0$
such that, if $0\leq \theta\leq \theta_0$, $(\Gamma,\ell_\theta)$ is a Heisenberg Uniqueness Pair.
\end{proposition}

\begin{proof}
Let $\gamma(s)=\bigl(s,\psi(s)\bigr)$, $\theta\in(-\pi/2,\pi/2]$. Then $\chi(s):=\pi_\theta\gamma(s)=s\cos\theta+\psi(s)\sin\theta$.
We have to show that $\chi$ is one-to-one, but $\chi'(s)=\cos\theta+\psi'(s)\sin\theta$.
As $\sin\theta\geq0$, the convexity hypothesis on $\psi$ implies that $\psi'(s)\geq \psi'(0^-)$ for $s<0$
thus $\chi'(s)\geq \cos\theta+\psi'(0^-)\sin\theta$ while for $s>0$
$\chi'(s)\geq \cos\theta+\psi'(0^+)\sin\theta$ for $s>0$. Thus, if $0\leq \theta<\theta_0$, $\chi'(s)>0$ for all $s$
and $\pi_\theta\gamma=\chi$ is one-to-one.

In the notation of Corollary \ref{cor:2}, $I_0=\R$. The result follows.
\end{proof}

\begin{example} Let $\alpha>0$, $\Gamma=\{(s,\mbox{sign}(s)|s|^\alpha),s\in\R\}$.
If $\alpha\geq 1$ and $\theta\in[0,\pi/2]$ then $(\Gamma,\ell_\theta)$
is a Heisenberg Uniqueness Pair. 

Using invariance property (Inv2) and $T(x,y)=(y,x)$ we also get that $(\Gamma,\ell_\theta)$
is a Heisenberg Uniqueness Pair when $0<\alpha\leq 1$ and $\theta\in [-\pi/2,0]$, 
\end{example}

\subsection{The regularity of $\Phi_\theta$}
\label{rem:smooth}

The aim of this section is to establish the regularity of the map $\Phi_\theta$. This is only needed when
we investigate closed curves.

We will fix an integer $k\geq 2$. Let $\gamma$ be a $\cc^k$-smooth mapping $I\to\R^2$
and $\Gamma$ be the corresponding curve. We will assume that $\gamma'$ does not vanish.
Assume that for some $\theta_0$, the conditions of Corollary \ref{cor:2} are satisfied.

Note that $\Phi_\theta$ is defined implicitly as follows: let
$F(s,t,\theta):=\pi_\theta\gamma(s)-\pi_\theta\gamma(t)$ then
$$
\left\{\begin{matrix}\Phi_{\theta_0}(s_\pm^0)=s_\mp^0\\
F(s,\Phi_\theta(s),\theta)=0\end{matrix}\right.
$$
provided we know $\Phi_{\theta_0}(s^0_\pm)$ for some $s_\pm^0\in I_\pm$. To simplify notation, we will only consider the case $\Phi_{\theta_0}(s_-^0)=s_+^0$.
Note that $\dst\frac{\partial F}{\partial t}=-\pi_\theta\gamma'(t)$ and $\dst\frac{\partial F}{\partial \theta}=F(s,t,\theta^\perp)$.

Using the Implicit Function Theorem, we deduce that

--- if $\pi_\theta\gamma'(s_+^0)\not=0$ then, for fixed $\theta=\theta_0$, $\Phi_\theta$ is well defined and of class $\cc^k$ in a neighborhood of $s_-^0$
and
\begin{equation}
\label{eq:diffphis}
\frac{\partial \Phi_{\theta_0}}{\partial s}(s_-)=\frac{\pi_{\theta_0}\gamma'(s_-)}{\pi_{\theta_0}\gamma'(s_+)}.
\end{equation}

--- if $\gamma'(s_-^0)\not=\gamma'(s_+^0)$ then $F(s_0,\Phi_{\theta_0}(s_0),\theta_0^\perp)\not=0$ thus, 
for $s_-^0$ fixed, there is a neighborhood of $\theta_0$ on which
$\Phi_\theta(s_-^0)$ is well defined and of class $\cc^\infty$ in $\theta$ and 
\begin{equation}
\label{eq:diffphitheta}
\frac{\partial \Phi_\theta}{\partial \theta}(s_-^0)=-\frac{\pi_\theta\gamma'(s_-^0)}{F(s_-^0,\Phi_\theta(s_-^0),\theta^\perp)}.
\end{equation}

--- if both conditions are satisfied, then $\Phi_\theta$ is defined and of class $\cc^1$ in the variables $(s_-,\theta)$ and
the derivatives are given by \eqref{eq:diffphis}-\eqref{eq:diffphitheta}.


Let us now assume that the curvature of $\Gamma$ does not vanish around $s_+^0$, thus if $\pi_\theta\gamma'(s_+^0)=0$,
$\pi_\theta\gamma''(s_+^0)\not=0$. First, let $s(\theta)$ be defined by $s(\theta_0)=s_+^0$ and
$\pi_\theta\gamma'\bigl(s(\theta)\bigr)=0$. From the Implicit Function Theorem, $s(\theta)$ is well defined and of class $\cc^k$
in a neighborhood of $\theta_0$ with $s'(\theta)=-\dst\frac{\pi_{\theta^\perp}\gamma'(s)}{\pi_\theta\gamma''(s)}$.

Further, assume that $s_-^0=s_+^0:=s^0$ that $I^0_-=(a,s^0)$ and $I_+^0=(s^0,b)$ then $\Phi_\theta(s)\,:\bigl(a,s(\theta)\bigr)\to \bigl(s(\theta),b\bigr)$.
Further,
$$
\gamma\bigl(s(\theta)+s\bigr)=\gamma\bigl(s(\theta)\bigr)+s\pi_{\theta^\perp}\gamma'\bigl(s(\theta)\bigr)
+\frac{s^2}{2}\ent{\pi_{\theta^\perp}\gamma"\bigl(s(\theta)\bigr)+\pi_{\theta}\gamma"\bigl(s(\theta)\bigr)}+o(s^2)
$$
where the $o(s^2)=s^2\chi_\theta(s)$ with $\chi_\theta(s)\to 0$ uniformly in $\theta$ in a compact neighborhood of $\theta_0$
(since $s(\theta)$ is smooth). 
We will now appeal to the following simple lemma. The proof is a classical exercise on Taylor expansions:

\begin{lemma}
\label{lem:reginv}
Let $\gamma\,:V\to\R^2$ be of class $\cc^k$ in a neighborhood $V$ of $0$. Assume that the Taylor expansion of $\gamma$ is of the form
$\gamma(s)=(a_0+a_1s+a_2s^2+\cdots+a_ks^k,b_0+b_2s^2+\cdots+b_ks^k)+o(s^k)$. Then there is a neighborhood $W$ of $0$
such that $\gamma$ is two-to-one on $W$: if $s\in W$, there is exactly one $\ffi(s)\in W$ such that $\ffi(s)\not=s$ and 
$\gamma\bigl(\ffi(s)\bigr)=\gamma(s)$. Moreover, the map $s\to\ffi(s)$ is of class $\cc^{k-1}$ with $\ffi(s)=-s+o(s)$.
\end{lemma}

Applying this lemma in the basis $(\theta,\theta^\perp)$ and at the point $s(\theta)$ instead of the standard basis and the point $0$, we obtain that $\Phi_\theta\bigl(s(\theta)+s\bigr)=s(\theta)-s+o(s)$ with a $o(s)$ that is uniform in $\theta$.
Therefore, $\Phi_\theta$ is of class $\cc^1$ both in $s$ and $\theta$ in a neighborhood of $s_0,\theta_0$.
If we use the fact that $\gamma$ has a Taylor expansion of order $k$, Lemma \ref{lem:reginv} implies that $\Phi_\theta$ is of class $\cc^{k-1}$
both in $s$ and $\theta$.

An example in which the above setting is satisfied is when $\Gamma$ is a closed convex curve with non vanishing curvature.
As this is the only instance in which we will appeal to the regularity of $\Phi_\theta$, let us summarize what we have just proved in this setting:

\begin{proposition}
\label{prop:reg}
Let $k\geq 2$ be an integer. Let $\Gamma$ be a closed convex $\cc^k$-smooth curve with non vanishing curvature and 
let $\gamma\,:[0,1)\to\R^2$ be a parametrization of $\Gamma$
such that $\gamma$ is of class $\cc^k$ and $\gamma'$ does not vanish.
For every $\theta\in\R$, let $\Phi_\theta\,:[0,1]\to[0,1)$ be the mapping that is given by
$\pi_\theta\gamma^{-1}\bigl(\gamma(s)\bigr)=\{s,\Phi_\theta(s)\}$ for $s\in[0,1]$ . Then $\Phi_\theta$ is well defined,
one-to-one on $[0,1)$, of class $\cc^{k-1}$ in $s$ and $\theta$.
\end{proposition}
%

\subsection{Two lines: a dynamical system approach}
\label{sec:general}

We will now consider Heisenberg Uniqueness Pairs of the form $(\Gamma,\ell_{\theta_1}\cup\ell_{\theta_2})$
where $\Gamma=\{\gamma(s),s\in I\}$ is a piecewise smooth curve, and $\theta_1\not=\theta_2\in[-\pi/2,\pi/2)$ are two angles.
Assume that for both angles, $\Gamma$ is as in Corollary \ref{cor:2}.
We thus have two splittings $I=I_0^1\cup I_+^1\cup I_-^1=I_0^2\cup I_+^2\cup I_-^2$
and two maps $\Phi_j=\Phi_{\theta_j}$, $j=1,2$. Write $\Gamma^j_\eps$, $j=1,2$, $\eps=0,+,-$ for the corresponding parts of $\Gamma$.

Let $\mu\in\aa\cc(\Gamma)$ and assume that $\widehat{\mu}(\xi)=0$ for $\xi\in\ell_{\theta_1}\cup\ell_{\theta_2}$.
According to Corollary \ref{cor:2}, $f_\mu=0$ on $I_0^1\cup I_0^2$. It follows that 
$\mu\in\aa\cc\bigl(\Gamma\setminus(\Gamma_0^1\cup \Gamma_0^2)\bigr)$.
Without loss of generality, we may now assume that $I_0^1=I_0^2=\emptyset$.

Since $\Phi_j$ is one to one, from \eqref{eq:symint} we deduce that, for every interval $J\subset I$,
\begin{equation}
\label{eq:symint2}
\int_{\Phi_j(J)}|f_\mu(s)|\,\mbox{d}s=\int_J|f_\mu(s)|\,\mbox{d}s.
\end{equation}

The fact that $(\Gamma,\ell_{\theta_1}\cup\ell_{\theta_2})$ is a Heisenberg Uniqueness Pair will depend on the properties of
the dynamical system generated by the map $\Phi=\Phi_{2}\circ\Phi_1$. We will denote by $\Phi^n$ the $n$-th iterate of $\Phi$.

We will now prove three lemmas that will allow us to establish Heisenberg Uniqueness.

\begin{lemma}
\label{lem:wandering}
Let $J\subset I$ be an interval and assume that $J$ is \emph{wandering} for $\Phi:=\Phi_2\circ\Phi_1$, that is for every $j\geq 1$
$\Phi^j(J)\cap J=\emptyset$ (up to a set of measure $0$). If $\mu\in\aa\cc(\Gamma)$ is such that
$\widehat{\mu}=0$ on $\ell_{\theta_1}\cup\ell_{\theta_2}$ then $f_\mu=0$ on $\dst\bigcup_{j=0}^\infty \Phi^j(J)$.
\end{lemma}

\begin{proof}
According to \eqref{eq:symint2}, for every interval $J$
$$
\int_{\Phi_2\bigl(\Phi_1(J)\bigr)}|f_\mu(s)|\,\mbox{d}s=\int_{\Phi_1(J)}|f_\mu(s)|\,\mbox{d}s
=\int_{J}|f_\mu(s)|\,\mbox{d}s
$$
and more generally, for every $k\geq0$
$$
\int_{\Phi^k(J)}|f_\mu(s)|\,\mbox{d}s=\int_{J}|f_\mu(s)|\,\mbox{d}s.
$$
But as the interval $J$ is wandering and $f_\mu\in L^1(I)$
\begin{eqnarray*}
+\infty&>&\int_{\bigcup_{j=0}^\infty \Phi^j(J)}|f_\mu(s)|\,\mbox{d}s=\sum_{j=0}^\infty\int_{\Phi^j(J)}|f_\mu(s)|\,\mbox{d}s\\
&=&\sum_{j=0}^\infty\int_{J}|f_\mu(s)|\,\mbox{d}s
\end{eqnarray*}
so that $0=\dst\int_{J}|f_\mu(s)|\,\mbox{d}s=\int_{\Phi^k(J)}|f_\mu(s)|\,\mbox{d}s$ thus $f_\mu=0$ on $\Phi^k(J)$.
\end{proof}

\begin{lemma}
\label{lem:attract}
Let $J\subset I$ be an interval and assume that $J$ is \emph{attractive} for $\Phi:=\Phi_2\circ\Phi_1$, that is,
there exists $k$ such that $\Phi^k(J)\subset J$. If $\mu\in\aa\cc(\Gamma)$ is such that
$\widehat{\mu}=0$ on $\ell_{\theta_1}\cup\ell_{\theta_2}$ then $\mbox{supp}\,f_\mu\cap J\subset\dst\bigcap_{n\geq 1}\Phi^{nk}(J)$.
\end{lemma}

\begin{proof} As in the previous proof
$$
\int_{J}|f_\mu(s)|\,\mbox{d}s=\int_{\Phi^k(J)}|f_\mu(s)|\,\mbox{d}s
$$
so that, if $\Phi^k(J)\subset J$, $f_\mu=0$ on $J\setminus \Phi^k(J)$. The result follows by noting that $\Phi^{(n+1)k}(J)\subset\Phi^{nk}(J)$.
\end{proof}

The last lemma only applies to closed curves. In this case, we can parametrize $\gamma$ with a function $\gamma\,:[0,1]\to\R^2$
with $\gamma(1)=\gamma(0)$ and $\gamma$ is one-to-one on $[0,1)$. A rotation of angle $\alpha$ is then the map 
$R_\alpha:t\to t+\alpha\,\mbox{mod}\,1$.

\begin{lemma}
\label{lem:rot}
Assume further that $\Gamma$ is a closed curve.
Assume that there is a $\cc^1$-diffeomorphism $h$ such that $\Phi$ is conjugated by $h$ to a rotation $R_\alpha$ with $\alpha\in\R\setminus\Q$:
$\Phi=h^{-1}\circ R_\alpha\circ h$.
Then $(\Gamma,\ell_{\theta_1}\cup\ell_{\theta_2})$ is a Heisenberg Uniqueness Pair.
\end{lemma}

\begin{proof} First note that, since $0\in\ell_{\theta_1}$, $\widehat{\mu}(0)=0$, that is
$$
\int_0^1 f_\mu(s)\,\mbox{d}s=0.
$$

As previously, but using \eqref{eq:symintsign} instead of
\eqref{eq:symint}, for every interval $I$,
$$
\int_{\Phi(I)}f_\mu(s)\,\mbox{d}s=\int_ If_\mu(s)\,\mbox{d}s.
$$
Thus, changing variable $s=h^{-1}(t)$ in both integrals we get
$$
\int_{h^{-1}(I)}\frac{f_\mu\bigl(h^{-1}(R_{-\alpha}t)\bigr)}{h'\bigl(h^{-1}(R_{-\alpha}t)\bigr)}\,\mbox{d}t
=\int_{R_\alpha h^{-1}(I)}\frac{f_\mu\bigl(h^{-1}(t)\bigr)}{h'\bigl(h^{-1}(t)\bigr)}\,\mbox{d}t
=\int_{h^{-1}(I)}\frac{f_\mu\bigl(h^{-1}(t)\bigr)}{h'\bigl(h^{-1}(t)\bigr)}\,\mbox{d}t.
$$
As this holds for every $I$, 
\begin{equation}
\label{eq:ergodic}
\frac{f_\mu\bigl(h^{-1}(R_{-\alpha}t)\bigr)}{h'\bigl(h^{-1}(R_{-\alpha}t)\bigr)}=
\frac{f_\mu\bigl(h^{-1}(t)\bigr)}{h'\bigl(h^{-1}(t)\bigr)}\quad a.e.
\end{equation}
But then
$$
\frac{f_\mu\bigl(h^{-1}(t)\bigr)}{h'\bigl(h^{-1}(t)\bigr)}=\frac{1}{n}\sum_{k=1}^n
\frac{f_\mu\bigl(h^{-1}(R_{-\alpha}^kt)\bigr)}{h'\bigl(h^{-1}(R_{-\alpha}^kt)\bigr)}
\to \mathcal{I}:=\int_0^1\frac{f_\mu\bigl(h^{-1}(t)\bigr)}{h'\bigl(h^{-1}(t)\bigr)}\,\mbox{d}t
$$
for almost every $t$, according to Birkhoff's Ergodic Theorem. Changing variable $s=h^{-1}(t)$
in the integral, we obtain that $\mathcal{I}=0$ so that $f_\mu=0$.
\end{proof}

\section{Heisenberg Uniqueness pairs obtained with the help of wandering sets}
\label{sec:wandering}
\subsection{Graphs of functions that go to $+\infty$ in $\pm\infty$}

\begin{theorem}\label{th:main}
Let $\psi$ be a piecewise smooth function and let $\theta_1\not=\theta_2\in(0,\pi)$ be such that

-- $\psi(s)\sin\theta_i+s\cos\theta_i\to+\infty$ when $s\to\pm\infty$

-- $\psi(s)\sin\theta_i+s\cos\theta_i$ has a unique local minimum.

Let $\Gamma=\bigl\{\bigl(s,\psi(s)\bigr),s\in\R\bigr\}$ be the graph of $\psi$. Then 
$(\Gamma,\ell_{\theta_1}\cup\ell_{\theta_2})$ is a Heisenberg uniqueness pair.
\end{theorem}

Before proving the theorem, Let us make a few comments on the hypothesis on $\psi$.

First, $(\Gamma,\ell_0)$ is a Heisenberg uniqueness pair according to corollary \ref{cor:2}. 

The following lemma shows that the requirements of Theorem \ref{th:main} are commonly met, in particular when $\psi$ is a polynomial of even degree.

\begin{lemma}\label{lem:triv}
Let $\psi$ be a function of class $\cc^1$ such that $\psi'$ satisfies the following two conditions:

-- $\psi'(t)\to-\infty$ when $t\to-\infty$ and $\psi'(t)\to+\infty$ when $t\to+\infty$;

-- $\psi'$ has only finitely many local extrema.

Then there exists $0<\theta_0<\theta_1<\pi$ such that, for $\theta\in(0,\pi)\setminus(\theta_0,\theta_1)$, 
the function $\chi$ defined by $\chi(t)=\psi(t)\sin\theta+t\cos\theta$ is such that
$\chi(t)\to+\infty$ when $t\to\pm\infty$ and $\chi$ has a unique local minimum.
\end{lemma}

\begin{proof}[Proof of Lemma \ref{lem:triv}] Note that, for $\theta\in(0,\pi)$, $\chi'(t)=\psi'(t)\sin\theta+\cos\theta\to\pm\infty$
when $t\to\pm\infty$. In particular, there exists $a>0$ such that, if $t>a$, $\chi'(t)\geq 1$. Therefore, for $t\geq a$,
$$
\chi(t)=\chi(a)+\int_a^t\chi'(s)\,\mbox{d}s\geq\chi(a)+t-a\to+\infty\mbox{ when }t\to+\infty.
$$
The proof that $\chi(t)\to+\infty$ when $t\to-\infty$ is similar.

Next, let $t_0<t_1<\cdots<t_N$ be the local extrema of $\psi'$. Then $\psi'$ is strictly increasing on $(-\infty,t_0)$ and on $(t_N,+\infty)$. Let $A=\max_{t\in[t_0,t_N]}|\psi'(t)|+1$ and $a$ be such that $\psi'(t)<-A$ on $(-\infty,a)$ and 
$\psi'(t)>A$ on $(a,+\infty)$. 
Let $\theta_0=\mathrm{arg\,cotan}A$ so that, if $0<\theta<\theta_0$, and $t>-a$
$$
\chi'(t)=\psi'(t)\sin\theta+\cos\theta\geq-A\sin\theta+\cos\theta>-A\sin\theta_0+\cos\theta_0=0.
$$
As $\chi'$ is continuous and $\chi'(t)\to-\infty$ when $t\to-\infty$, $\chi'$ vanishes at a unique point $t_\theta\in(-\infty,a)$
where it changes sign from negative to positive, therefore $\chi$ has a minimum at $t_\theta$.

Taking $\theta_1=\pi-\theta_0$, the same argument shows that there is a unique $t_\theta\in(a,+\infty)$ such that
$\chi'(t)<0$ for $t<t_\theta$ and $\chi'(t)>0$ for $t>t_\theta$ thus $\chi$ has a unique minimum at $t_\theta$.
\end{proof}

%
%

We are now in position to prove the theorem.

\begin{proof}[Proof of Theorem \ref{th:main}] As noticed above, the result is trivial if $\theta_1=0$ or $\theta_2=0$. We will thus assume
that $\theta_1,\theta_2\not=0$

Let $\gamma(s)=\bigl(s,\psi(s)\bigr)$ and let
$\sigma_0$ be the unique local minimum of $\scal{\gamma(s),\theta_1}=s\cos\theta_1+\psi(s)\sin\theta_1$.
Note that if $\gamma$ is smooth this is the unique point such that $\theta_1$ is normal to $\Gamma$
thus $\theta^\perp_1$ is tangent to $\Gamma$.
Without loss of generality, using the invariance property (Inv1),
we may assume that $\sigma_0=0$ and $\gamma(0)=(0,0)$. Using (Inv2) we may further assume that $\theta_1=-\pi/2$
so that $\theta_1^\perp=\vec{i}:=(1,0)$.

The hypothesis on $\Gamma$ and $\theta_1$ ensure that we may apply Corollary \ref{cor:2}.
In the notation of Section \ref{sec:general},
$I_0^1=\emptyset$, $I_-^1=(-\infty,0]$ and $I_+^1=[0,+\infty)$ and the map $\Phi_1\,:\R\to\R$
is the map such that, for every $s\not=0$, $s\Phi_1(s)<0$ and $\psi\bigl(\Phi_1(s)\bigr)=\psi(s)$.
Note that $\psi$ is decreasing on $I_-^1$ and increasing on $I_+^1$.

Now $\scal{\gamma(s),\theta_2}=s\cos\theta_2+\psi(s)\sin\theta_2$ has also a unique local minimum at
$s_2$. Up to a symmetry $T\,:(x,y)\to(-x,y)$, the invariance property (Inv2) shows that we may assume that $s_2\geq0$
(note that this implies that $0<\theta_2^\perp<\pi/2$).
Thus in the notation of Section \ref{sec:general},
$I_0^2=\emptyset$, $I_-^2=(-\infty,s_2]$ and $I_+^2=[s_2,+\infty)$ and the map $\Phi_2\,:\R\to\R$
is such that $\Phi_2(I_\pm^2)=I_\mp^2$ and 
$\Gamma\cap\bigl(\gamma(s)+\R\theta_2^\perp\bigr)=\bigl\{\gamma(s),\gamma\bigl(\Phi_2(s)\bigr)\bigr\}$.
Note that $\psi(s)+s\sin\theta_2$ is decreasing on $I_-^2$ and increasing for $I_+^2$.

Let us first assume that $s_2>0$ and let $\sigma_0=0$.

Next, define $\sigma_1=\Phi_2(\sigma)>s_2>0$ and, for $k\geq 1$, $\sigma_{k+1}=\Phi_2\bigl(\Phi_1(\sigma_k)\bigr)$.
We assert that $[\sigma_1,\sigma_2]$ is a wandering set for $\Phi=\Phi_2\circ\Phi_1$
and $\bigcup_{k\geq 1}[\sigma_k,\sigma_{k+1}]=[\sigma_1,+\infty)$.

Before proving this assertion, let us show that the conclusion of Theorem \ref{th:main} follows. Indeed, according to Lemma
\ref{lem:wandering}, $f_\mu=0$ on $[\sigma_1,+\infty)$. Appealing to Corollary \ref{cor:2} for $\Phi_2$, \eqref{eq:symf} reduces to
$f_\mu=0$ on $(-\infty,0]=\Phi_2^{-1}\bigl([\sigma_1,+\infty)\bigr)$ and then, appealing to Corollary \ref{cor:2} for $\Phi_1$,
\eqref{eq:symf} reduces to $f_\mu=0$ on $[0,+\infty)$ as well.

\begin{figure}[!ht]
\begin{center}
\setlength{\unitlength}{0.5cm}
\begin{picture}(16,8)
\qbezier[400](1,8)(8,-5)(15,8)
\dottedline{0.1}(8,1.5)(15.8,4.5)
\dottedline{0.1}(10.6,2.5)(5.2,2.5)
\dottedline{0.1}(5.2,2.5)(15.6,6.5)
\put(0,1.5){\vector(1,0){16}}
\put(15.5,1.7){$\theta_1^\perp$}
\put(8.1,7){$\theta_1$}
\put(8,0){\vector(0,1){8}}
\put(2.3,7){$t=\gamma(s)$}
\put(8.1,0.7){$\sigma_0$}
\put(10.9,2){$\sigma_1$}
\put(14.9,7){$\sigma_2$}
\end{picture}
\caption{The construction of $\sigma_k$}
\label{fig:cor3}
\end{center}
\end{figure}
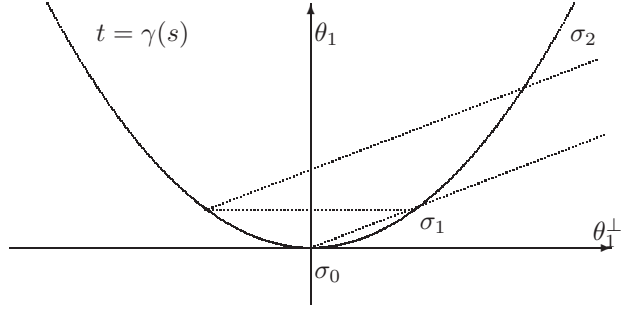

Let us now show that $\sigma_k$ is strictly increasing. This follows from a simple geometric consideration:
since $0<\theta_2^\perp<\pi/2$, if we start at a point $A\in\Gamma$ in the plane, moving horizontally to the left till we reach $\Gamma$ again
in some point $B$ and then to the right in direction $\theta_2^\perp$, we are moving upward and can therefore only reach $\Gamma$ again on the right of $A$.

More precisely, let $s_k:=\Phi_1(\sigma_k)<0<\sigma_k$ so that $\psi(s_k)=\psi(\sigma_k)$. Then, for $t>0$ if $s=s_k+t(\sigma_k-s_k)$,
$$
\psi(s)\begin{cases}<\psi(s_k)&\mbox{if }0<t<1\\
=\psi(\sigma_k)&\mbox{if }t=1\\
>\psi(s_k)&\mbox{if }t>1
\end{cases}
$$
since $\psi$ decreases on $[s_k,0]$ and increases on $\R^+$.
On the other hand, for $t>0$ $\psi(s_k)+t(\sigma_k-s_k)\sin\theta_2>\psi(s_k)=\psi(\sigma_k)$. Thus if $\psi(s_k)+t(\sigma_k-s_k)\sin\theta_2=\psi(s)$
then $t>1$ that is $s>\sigma_k$. But, by definition, $s=\Phi_2(s_k)=\Phi_2\bigl(\Phi(\sigma_k)\bigr)=\sigma_{k+1}$.

Finally, the only possible finite limit
of $\sigma_k$ is a fix point of $\Phi_2\circ\Phi_1$ that is $0$. As $\sigma_k>\sigma_1>s_2\geq0$, this is not possible.

In the case $s_2=0$, it is enough to take $\sigma_0<0$ and then $\sigma_1=\Phi_2(\sigma_0)>0$. The same reasoning works and shows that $f_\mu=0$
on $\R\setminus[\sigma_0,\sigma_1]$. But as $\sigma_0$ is arbitrary, we let $\sigma_0\to 0$ and $s_2=0$ implies $\sigma_1\to0$ as well.
\end{proof}

\subsection{Cusps}

\begin{proposition}
\label{lem:cusp}
Let $\psi\,:\R\to\R$ be a function such that 
\begin{enumerate}
\renewcommand{\theenumi}{\alph{enumi}}
\item $\psi$ is continuous, smooth on $\R\setminus\{0\}$, 

\item $\psi(0)=0$ and when $t\to\pm\infty$, $\psi(t)\to+\infty$ while $\frac{\psi(t)}{t}\to 0$,

\item $\psi$ is strictly convex on $(-\infty,0)$ and strictly concave on $(0,+\infty)$.
\end{enumerate}
Let $\Gamma=\bigl\{\bigl(t,\psi(t)\bigr),t\in\R\bigr\}$ be the graph of $\psi$.
\begin{enumerate}
\renewcommand{\theenumi}{\roman{enumi}}
\item\label{cusp1} If $\psi$ has a left or right derivative at $0$ then there is a set $\mathcal{O}$
with $0\in \mathcal{O}$ and non-empty interior such that, for $\theta\in O$, $(\Gamma,\ell_\theta)$ is a Heisenberg Uniqueness Pair.

\item\label{cusp2} For every $\theta_1\in(-\pi/2,\pi/2)$ there is a non-empty open set $O(\theta_1)$ such that,
if $\theta_2\in O(\theta_1)$, then $(\Gamma,\ell_{\theta_1}\cup\ell_{\theta_2})$ is a Heisenberg Uniqueness Pair.
\end{enumerate}
\end{proposition}

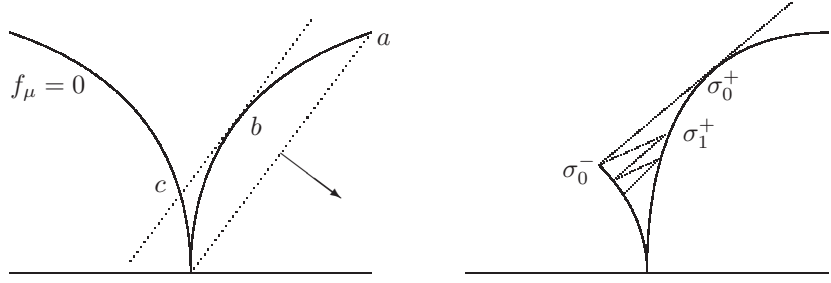
\begin{figure}[!ht]
\begin{tabular}{cc}
\setlength{\unitlength}{0.8cm}
\begin{picture}(7,4.5)
\dottedline{0.01}(0,0)(6,0)
\qbezier(3,0)(3,3)(6,4)
\qbezier(3,0)(3,3)(0,4)
\dottedline{0.1}(3,0)(6,4)
\put(4.5,2){\vector(4,-3){1}}
\put(4,2.3){$b$}
\put(6.1,3.8){$a$}
\put(2.4,1.35){$c$}
\dottedline{0.1}(2,0.2)(5,4.2)
\put(0,3){$f_\mu=0$}
\end{picture}
&
\setlength{\unitlength}{0.8cm}
\begin{picture}(6,4.5)
\dottedline{0.01}(0,0)(6,0)
\qbezier(3,0)(3,4)(6,4)
\qbezier(3,0)(3,1)(2.2,1.8)
\put(4,3){$\sigma_0^+$}
\put(3.6,2.2){$\sigma_1^+$}
\put(1.6,1.6){$\sigma_0^-$}
\dottedline{0.05}(2.2,1.8)(5.4,4.5)
\dottedline{0.05}(2.2,1.8)(3.3,2.3)
\dottedline{0.05}(3.3,2.3)(2.45,1.55)
\dottedline{0.05}(2.45,1.55)(3.2,1.9)
\dottedline{0.05}(3.2,1.9)(2.6,1.3)
\end{picture}
\end{tabular}
\caption{The construction of $\sigma_k^\pm$}
\end{figure}

\begin{proof}  Let $\gamma(t)=\bigl(t,\psi(t)\bigr)$ and $\mu\in\aa\cc(\Gamma)$. 
The main difficulty here is that a line may intersect $\Gamma$
up to three times. At this stage we have not been able to fully characterize the set
of lines that lead to Heisenberg Uniqueness Pairs. In order to prove the proposition,
it is enough to consider angles in $(-\pi/2,0)$. The invariance property (Inv2) for the map $T(x,y)=(-x,y)$
will give the result for positive angles.

The proof of \ref{cusp1}) as well as $\theta_1=0$ in \ref{cusp2}) is similar to
Proposition \ref{prop:1line}. Define $\theta_+=\arctan\gamma'(0^+)-\pi/2$ ($\theta_+=0$ if $\gamma'(0)=+\infty$)
and let $\theta\in[\theta_+,0]$. The convexity properties of $\psi$ imply that, for every $s\in\R$,
$\pi_\theta\gamma^{-1}\bigl(\gamma(s)\bigr)=\{\gamma(s)\}$. We may thus apply Corollary \ref{cor:2}.
In the notation of this corollary $I_0=\R$ so that, if $\mu\in\aa\cc(\Gamma)$ is such that $\widehat{\mu}=0$
on $\ell_\theta$, then $f_\mu=0$ thus $\mu=0$.

Now let $-\pi/2<\theta<\theta_+$ and let us first investigate
what happens when $\widehat{\mu}=0$ on $\ell_\theta$. $\Gamma$ can then be divided into 3 parts:
Define $a=a(\theta)$ by $\pi_\theta\gamma^{-1}(0)=\{0,\gamma(a)\}$. Define $b=b(\theta)$ as the unique $s\in\R$ such that
$\scal{\gamma'(b),\theta}=0$ {\it i.e.} $\theta_1^\perp$ is tangent to $\Gamma$ at $\gamma(b)$.
Let $c=c(\theta)$ be defined by
$\pi_\theta\gamma^{-1}\bigl(\gamma(b)\bigr)=\bigl\{\gamma(b),\gamma(c)\bigr\}$.
Note that $c(\theta)<0<b(\theta)<a(\theta)$, $a(\theta),b(\theta)$ are decreasing on $(-\pi/2,\theta_+)$ with
$a(\theta),b(\theta)\to+\infty$ when $\theta\to-\pi/2$
while $c(\theta)$ is decreasing with $c(\theta)\to-\infty$ as $\theta\to-\pi/2$.
Further notice that, the convexity properties of $\psi$ imply that, for $s<c$ and for $s>a$,
$\pi_\theta\gamma^{-1}\bigl(\gamma(s)\bigr)=\{\gamma(s)\}$. 
In particular, \eqref{eq:fund} reduces to $f_\mu(s)=0$.
On the other hand,
for $s\in\bigl(c(\theta),0\bigr)$ there exists a unique
$\Phi_\theta(s)\in\bigl(0,b(\theta)\bigr)$ and a unique $\Psi_\theta(s)\in\bigl(b(\theta),a(\theta)\bigr)$
such that
$$
\pi_\theta\gamma^{-1}\bigl(\gamma(s)\bigr)=\bigl\{\gamma(s),\gamma\bigl(\Phi(s)\bigr),\gamma\bigl(\Psi(s)\bigr)\bigr\}.
$$
The maps $\Phi_\theta,\Psi_\theta$ are also onto.

Now, fix $\theta_1\in(-\pi/2,0)$ and let $\theta_-$ be such that $b(\theta)\geq a(\theta_1)$ if $-\pi/2\leq\theta\leq \theta_-$
and let $\theta_2\in[-\pi/2,\theta_-]$ (note that $\theta_2<\theta_1$). Write $\Phi_j,\Psi_j$ for $\Phi_{\theta_j},\Psi_{\theta_j}$.
Note that, for $s\in\bigl[c(\theta_2),0\bigr)$, $\Psi(s)>b(\theta_2)>a(\theta_1)$.

Assume now that $\mu\in\aa\cc(\Gamma)$ is such that 
$\widehat{\mu}=0$ on $\ell_{\theta_1}\cup\ell_{\theta_2}$. Let $\sigma_0^+=a(\theta_1)$ and $\sigma_0^-=c(\theta_1)$.
As noticed above, $f_\mu(s)=0$ for $s>\sigma_0^+$ and $s<\sigma_0^-$. 
Let $\sigma_1^+=\Phi_2(\sigma_0^-)$ and note that $\sigma_1^+<b(\theta_1)$.
Then \eqref{eq:fund} for $\theta=\theta_2$ and $s<\sigma_0^-$ reduces to $f_\mu\bigl(s)=0$ for $s>\sigma_1^+$.

Next define inductively $\sigma_k^-=\Phi_1^{-1}(\sigma_k^+)$ and $\sigma_{k+1}^+=\Phi_2(\sigma_k^-)$
so that $\sigma_k^-$ is increasing, $\sigma_k^+$ is decreasing, both have limit $0$.
Moreover, applying \eqref{eq:fund} for $\theta=\theta_1$ and $s\in(\sigma_{k-1}^+,\sigma_k^+)$
shows that $f_\mu(s)=0$ for $s\in(\sigma_k^-,\sigma_{k-1}^-)$ and applying \eqref{eq:fund} for $\theta=\theta_2$
shows that  $f_\mu(s)=0$ for $s\in(\sigma_k^+,\sigma_{k+1}^+)$.
\end{proof}

\subsection{The graph of $t\to |t|^\alpha$, $\alpha>0$}

We can now prove point \ref{mt2} of the main theorem.

\begin{theorem}
\label{cor:thmain}
Let $p\geq 1$ and $\Gamma=\{(s,|s|^p),\ s\in\R\}$. There exists a set $E\subset[-\pi/2,\pi/2)^2$ of positive measure
such that, if $(\theta_1,\theta_2)\in E$, 
$(\Gamma,\ell_{\theta_1}\cup\ell_{\theta_2})$ is a Heisenberg Uniqueness Pair.
\end{theorem}

The case $p=2$ is due to P. Sj\"olin \cite{Sjo} and the proof of Theorem \ref{th:main} is inspired by his work.
The case $p>1$ is covered by Theorem \ref{th:main} and in this case any pair $\theta_1\not=\theta_2\in [-\pi/2,\pi/2)^2$
will work.
The case $0<p<1$ is covered by Proposition \ref{lem:cusp}. At this stage we do not have a precise description of $E$ 
and we postpone it to future work.
It remains to prove the case $p=1$. We will show that again any pair $\theta_1\not=\theta_2\in [-\pi/2,\pi/2)^2$
will work.

\begin{proof} 
 For $|\theta-\pi/2|>\pi/4$, Corollary \eqref{cor:2} shows that 
$(\Gamma,\ell_\theta)$ is a Heisenberg Uniqueness Pair since then $I_0=\R$.

If $|\theta_1-\pi/2|,|\theta_2-\pi/2|<\pi/4$, then we may again apply Theorem \ref{th:main}.

It remains to consider the case $\theta_1=\pi/4$ or $3\pi/4$ and $|\theta_2-\pi/2|\leq\pi/4$.
We will only consider the case $\theta_1=\pi/4$, the other case being similar.

Let $\gamma(s)=|s|$, and $\mu=f_\mu\,\mbox{d}s$. Write $f_\mu^\pm$ for the restriction of
$f_\mu$ to $\R^\pm$. If $\widehat{\mu}=0$ on $\ell_{\pi/4}$ then,
for every $t\in\R$
$$
0=\int_\R f_\mu(s)e^{-i(s+|s|)t/\sqrt{2}}\,\mbox{d}t
=\int_{\R^-}f_\mu^-(s)\,\mbox{d}s+\widehat{f_\mu^+}(\sqrt{2}t).
$$
Thus $\dst\int_{\R^-}f(s)\,\mbox{d}s=-\widehat{f_\mu^+}(\sqrt{2}t)$. Riemann-Lebesgue's Lemma then implies that
$\dst\int_{\R^-}f_\mu^-(s)\,\mbox{d}s=0$ thus $\widehat{f_\mu^+}(\sqrt{2}t)=0$ thus $f_\mu^+=0$.
Now, if $\theta_2\not=\pi/4$ and $\widehat{\mu}=0$ on $\ell_{\theta_2}$ then
$$
0=\int_{\R^-}f_\mu(s)e^{-i(\cos\theta_2-\sin\theta_2)st}\,\mbox{d}s
=\widehat{f_\mu}\bigl((\cos\theta_2-\sin\theta_2)t\bigr)
$$
for every $t\in\R$ and as $\cos\theta_2-\sin\theta_2\not=0$, $f_\mu=0$.
\end{proof}

\subsection{Hyperbolas}

Let $\Gamma$ be the hyperbola  
$$
\Gamma=\{(u,v)\in \R^2 \,:\ v^2-u^2=1\}.
$$ 
Let $I=(0,1/2)\cup(1/2,1)$ and $\gamma:I\to \R^2$ be a parametrization of $\Gamma$ given by 
$$
\gamma(s)=(\cotan (2\pi s), 1/\sin(2\pi s)).
$$

\begin{theorem}
\label{th:hyp}
Let $\Gamma$ be the hyperbola  $\Gamma=\{\gamma(s),s\in I\}$. Then $(\Gamma,\ell_{\pm\pi/4})$ is a Heisenberg Uniqueness Pair. Moreover if $\theta_1\neq \theta_2 \in (-\pi/2,\pi/2)$. Then $(\Gamma,\ell_{\theta_1} \cup\ell_{\theta_2})$ is  a Heisenberg Uniqueness Pair if and only if $\theta_1\not\perp\theta_2$. 
\end{theorem}

\begin{remark}
In \cite{HMR}, the authors give a necessary and sufficient conditions for a lattice cross $\Lambda$ in $\ell_{+\pi/4}\cup\ell_{-\pi/4}$
to form a Heisenberg uniqueness pair $(\Gamma,\Lambda)$.
\end{remark}

\begin{proof} If $\theta=\pm\dst\frac{\pi}{4}$ than any line orthogonal to $\theta$ intersects $\Gamma$ in at most one point.
That is, in the notation of Corollary \ref{cor:2} $I_0=I$ and the theorem follows.

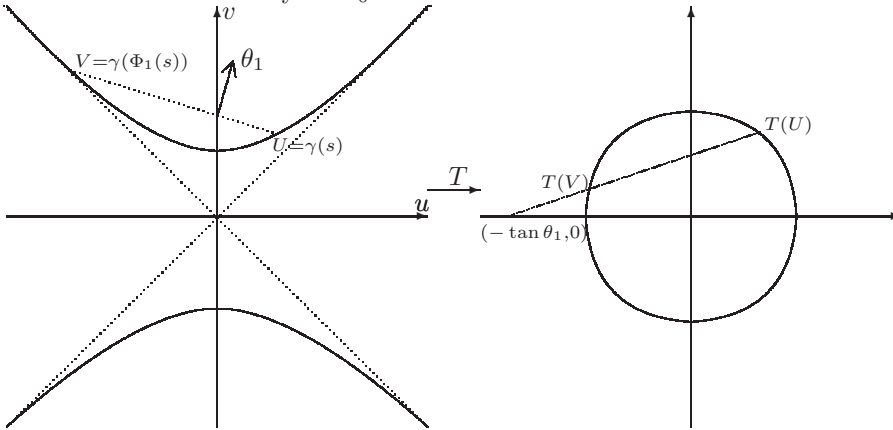
\setlength{\unitlength}{0.35cm}
\begin{picture}(34,16)
\put(8,0){\vector(0,1){16}}
\put(8.2,15.5){$v$}
\put(0,8){\vector(1,0){16}}
\put(15.5,8.2){$u$}
\dottedline{0.2}(-0.1,16)(16,0)
\dottedline{0.2}(0,0)(16.1,16)
\qbezier[400](0,16)(5,10.5)(8,10.5)
\qbezier[400](16,16)(11,10.5)(8,10.5)
\qbezier[400](0,0)(5,4.5)(8,4.5)
\qbezier[400](16,0)(11,4.5)(8,4.5)
\dottedline{0.2}(10.1,11.2)(2.6,13.5)
\put(10.1,10.5){$\scriptstyle U=\gamma(s)$}
\put(2.6,13.7){$\scriptstyle V=\gamma(\Phi_1(s))$}
\dottedline{0.05}(8,11.8)(8.6,13.9)
\dottedline{0.05}(8.7,13.3)(8.6,13.9)
\dottedline{0.05}(8.2,13.5)(8.6,13.9)
\put(8.9,13.7){$\theta_1$}
\put(16,9){\vector(1,0){2}}
\put(16.8,9.2){$T$}
\put(26,0){\vector(0,1){16}}
\put(18,8){\vector(1,0){16}}
\put(15.5,8.2){$u$}
\qbezier[400](22,8)(22.2,11.8)(26,12)
\qbezier[400](22,8)(22.2,4.2)(26,4)
\qbezier[400](30,8)(29.8,4.2)(26,4)
\qbezier[400](30,8)(29.8,11.8)(26,12)
\dottedline{0.1}(19,8)(28.6,11.2)
\put(28.8,11.3){$\scriptstyle T(U)$}
\put(20.3,9.1){$\scriptstyle T(V)$}
\put(18,7.2){$\scriptstyle(-\tan\theta_1,0)$}
\end{picture}

Let $\Phi_j=\Phi_{\theta_j}$, $j=1,2$  be the maps defined in Section \ref{sec:general}.
Consider the transformation\footnote{This transformation has a natural interpretation in projective geometry.} $T:(u,v)\to (u/v,1/v)$. 
Notice that $T$ is a one-to-one map from $\Gamma$ onto the circle  $\T_*=\S^1\backslash \{(-1,0),(1,0)\}$ and moreover the
image of any line orthogonal to $\theta$ is a line $L$ through the point $(-\tan \theta,0)$.\\

Let $\widetilde{\Phi}_j:\T_*\to \T_*$ be a map defined as follows: $\widetilde{\Phi}_j(\alpha)$ is the unique 
$\beta\in \T_*\backslash\{\alpha\}$ such that the line $L_{(\alpha,\beta)}$ joining $\alpha$ and $\beta$ contains the point $(-\tan \theta_j,0)$. Note that 
$$
T(\gamma(\Phi_j(s)))=\widetilde{\Phi}_j(T(\gamma(s))),\qquad j=1,2.
$$
This transformation allows to transfer the dynamical system generated by $\Phi=\Phi_1\circ\Phi_2$ on $\Gamma$ to a new dynamical system on $\T_*$
generated by
$\widetilde{\Phi}=\widetilde{\Phi}_1\circ\widetilde{\Phi}_2$. 
In particular we will cover $\T_*$ by wandering sets for $\widetilde{\Phi}$.  As  $T$ is a bijection, $I$ will thus be covered
by wandering sets for $\Phi$. The theorem then follows from Lemma \ref{lem:wandering}.

We distinguish two cases:

\begin{center}
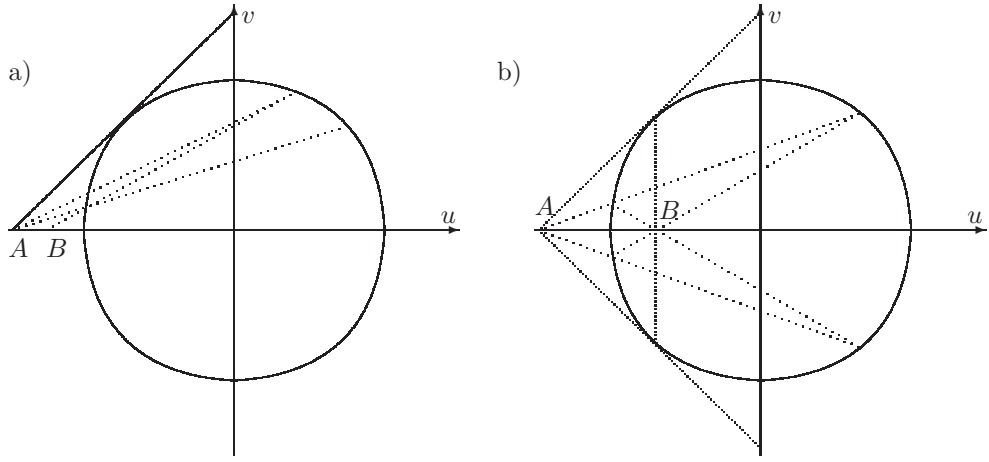
\begin{figure}[!ht]
\setlength{\unitlength}{0.5cm}
\begin{picture}(26,12)
\put(6,0){\vector(0,1){12}}
\put(6.2,11.5){$v$}
\put(0,6){\vector(1,0){12}}
\put(11.5,6.2){$u$}
\put(0,10){a)}
\put(13,10){b)}
\qbezier[400](2,6)(2.2,9.8)(6,10)
\qbezier[400](2,6)(2.2,2.2)(6,2)
\qbezier[400](10,6)(9.8,9.8)(6,10)
\qbezier[400](10,6)(9.8,2.2)(6,2)
\dottedline{0.02}(0.1,6)(6,11.8)
\dottedline{0.2}(0.1,6)(7.5,9.6)
\dottedline{0.2}(1,6)(7.5,9.6)
\dottedline{0.2}(0.1,6)(8.8,8.7)
\put(0,5.3){$A$}
\put(1,5.3){$B$}
\qbezier[400](16,6)(16.2,9.8)(20,10)
\qbezier[400](16,6)(16.2,2.2)(20,2)
\qbezier[400](24,6)(23.8,9.8)(20,10)
\qbezier[400](24,6)(23.8,2.2)(20,2)
\put(20,0){\vector(0,1){12}}
\put(20.2,11.5){$v$}
\put(14,6){\vector(1,0){12}}
\put(25.5,6.2){$u$}

\dottedline{0.1}(14.1,6)(20,11.8)
\dottedline{0.1}(14.1,6)(20,0.2)
\put(14,6.4){$A$}
\dottedline{0.1}(17.2,9)(17.2,3)
\put(17.3,6.3){$B$}

\dottedline{0.2}(14.1,6)(22.6,9.1)
\dottedline{0.2}(14.1,6)(22.6,2.9)
\dottedline{0.2}(16.2,6.6)(22.6,2.9)
\dottedline{0.2}(16.2,5.4)(22.6,9.1)
\end{picture}
\caption{a) $\theta_1>\pi/4$ and
b) $0<\theta_1<\pi/4$ ,$\theta_2=\theta_1-\frac{\pi}{2}$}
\end{figure}
\end{center}

\smallskip

\noindent{\bf First case.} $\theta_1\not\in (-\pi/4,\pi/4)$\\

\smallskip

Using the invariance property (Inv2), we assume without loss of generality  that $\theta_1\in(\pi/4,\pi/2)$
and that $\theta_2<\theta_1$. Thus $\tan \theta_1>1$ and $\tan \theta_2<\tan\theta_2$.
Let $A=(-\tan\theta_1,0)$ and note that, since $|\tan \theta_1|>1$, $A$ is in the ``exterior'' of $\T_*$. Let 
$$
\T_\pm=\{z\in \T_*\text{ : }\pm \Im z>0\}
$$
and $\alpha_\pm\in \T_\pm$ the unique point $\alpha\in \T_\pm$ such that the line $L_{(\alpha,A)}$ is tangent to $\T$. Note that $\Re \alpha_\pm =-1/\tan 
\theta_1$. For $\epsilon_1,\epsilon_2=\pm 1$, let 
$$
\T_{\epsilon_1,\epsilon_2}=\{z\in \T_{\epsilon_2}\text{ : }\epsilon_1\Re z>-1/\tan\theta_1\}.
$$
Note that $\widetilde{\Phi}_1$ is a one-to-one from $\T_{\epsilon_1,\epsilon_2}$
onto  $\T_{-\epsilon_1,\epsilon_2}$ while $\widetilde{\Phi}_2$ is a one-to-one from $\T_{\epsilon}$ onto  $\T_{-\epsilon}$.

We need the following observation. Let $\alpha \in \T_{-,+}$ and $\beta =\widetilde{\Phi}(\alpha)$ and let $\varphi$ ---resp. $\psi$--- be the angle between the real axis and the line $L_{(A,\alpha)}$ ---resp. $L_{(B,\alpha)}$. 
\begin{itemize}
\item If $-\tan \theta_1<-\tan \theta_2<-1/\tan \theta_1$,  then $|\varphi|>|\psi|$;

\item  if   $-\tan \theta_2=-1/\tan \theta_1$, hence $\theta_1\perp\theta_2$  and then  $\varphi=-\psi$; 

\item if  $\tan\theta_2>-1/\tan\theta_1$ then $|\varphi|<|\psi|$. 
\end{itemize}

So if $\theta_1\not\perp\theta_2$, then  the absolute value of the angles between the real axis and $L_{(A,\widetilde{\Phi}^k(\alpha))}$ is strictly monotonic. 
Therefore $[\alpha,\widetilde{\Phi}(\alpha)]$ is wandering.

\smallskip

If  $\theta_1\perp\theta_2$ then $\widetilde{\Phi}^k(\alpha)$ is $2$--periodic. Since $\theta_1\in(\pi/4,\pi/2)$, we can define 
$$
x_{\theta_1}=\frac{1}{\sqrt{\sin^2\theta_1-\cos^2\theta_1}}(-\cos\theta_1,\sin\theta_1)\in\Gamma.
$$
Note that this point is defined by $\alpha_+=T(x_{\theta_1})$. Let $s_0\in[0,1/2)$ be such that $x_{\theta_1}=\gamma(s_0)$
and $\Gamma_0=\{\gamma(s),\ s\in(0,s_0]\}$. Let $f$ be any function $f\in L^1(\Gamma_0)$. We will now extend $f$ to $L^1(\Gamma)$ as follows:

-- first, for $s_+^1\in(s_0,1/2)$, there is a unique $s_-^1\in(0,s_0)$ such that the line joining $\gamma(s_-^1)$ to $\gamma(s_+^1)$ 
is orthogonal to $\theta_1$ and we define $f(s_+^1)$ via Equation \eqref{eq:symf} \emph{for $\theta_1$}:
\begin{equation}
\label{eq:symhyp1}
\frac{f(s_+^1)}{\pi_{\theta_1}\gamma'(s_+^1)}=\frac{f(s_-^1)}{\pi_{\theta_1}\gamma'(s_-^1)}.
\end{equation}

-- Next, for every $s_+^2\in (1/2,1)$ there is a unique $s_-^2\in(0,1/2)$ such that the line joining $\gamma(s_-^2)$ to $\gamma(s_+^2)$ 
is orthogonal to $\theta_2$ and we define $f(s_+^2)$ via Equation \eqref{eq:symf} \emph{for $\theta_2$}:
\begin{equation}
\label{eq:symhyp2}
\frac{f(s_+^2)}{\pi_{\theta_2}\gamma'(s_+^2)}=\frac{f(s_-^2)}{\pi_{\theta_2}\gamma'(s_-^2)}.
\end{equation}
We will denote by $s_1=(s_0)_+^2$.

-- Finally, for $s_+^1\in(1/2,s_1)$, there is a unique $s_-^1\in(s_1,1/2)$ such that the line joining $\gamma(s_-^1)$ to $\gamma(s_+^1)$ 
is orthogonal to $\theta_1$ and one easily checks that \eqref{eq:symhyp1} is satisfied.

Let $\mu=f\,\mbox{d}s\in\aa\cc(\Gamma)$. According to Corollary \ref{cor:2}, $\widehat{\mu}=0$ on $\ell_{\theta_1}\cup\ell_{\theta_2}$.
Moreover, every $\mu\in\aa\cc(\Gamma)$ such that $\widehat{\mu}=0$ on $\ell_{\theta_1}\cup\ell_{\theta_2}$ can be constructed this way.

\smallskip

{\bf Second case.} $\theta_1,\theta_2\in (-\pi/4,\pi/4)$.

\smallskip

Without loss of generality, we assume that $\theta_2<\theta_1$ thus $-1<\tan \theta_2<\tan \theta_1<1$.   Let 
$$
\{\alpha_{\pm}\}=\{z\in \T_{\pm}\text{ : } \Re z=-\tan \theta_1\}
$$
be two point at the vertical of $A$ and define $\T_{\pm,\pm}$ as previously.
Let $\alpha\in \T_{+,+}$ and let $\varphi$ (resp. $\psi$) be the angle of $L_{(\alpha,A)}$ (resp. $L_{(\widetilde{\Phi}(\alpha),A)}$)  with the real 
axis, then $\psi>\varphi$. Again this implies that $[\alpha,\varphi(\alpha)]$ is wandering. 
The cases $\alpha\in\T_{+,-}$, $\T_{-,+}$, $\T_{-,-}$ are similar.

\end{proof}

\subsection{Closed convex curves with a corner point}
Let $\Gamma=\{\gamma(s),s\in[0,1]\}$ be a closed \emph{convex} curve and assume that $\gamma$ is piecewise smooth $1$-periodic and that $\gamma'$
has a jump singularity at $0$ {\it i.e.} $\Gamma$ has a corner point at $\gamma(0)$. Without loss of generality $\gamma(0)=0$
and let $\gamma^{\prime}(0^\pm)$ the vectors defining the two half tangents to $\Gamma$ at $0$. 
Let $H_0$ be a supporting hyperplane of $\Gamma$ at $0$. As $0$ is a corner point of $\Gamma$, this supporting 
line is not unique
and we may assume that $H_0\cap\Gamma=\{0\}$. Up to a rotation, we may assume that $H_0$ is the vertical axis.
Up to a symmetry, we may also assume that $\gamma$ covers $\Gamma$ in counter-clockwise order.

The fact that $0$ is a corner point implies that $H_0$ and $\gamma^{\prime}(0^\pm)$
define two positive open cones $\cc_\pm$ with $\cc_+$ in the upper half right quadrant and $\cc_-$ in the lower half right quadrant.
Let $\cc^*_\pm$ be the dual cone of $\cc_\pm$ ({\it i.e.} $\theta^\perp\in\cc_+$ if and only if $\theta\in\cc_+^*$).

Let $\theta_1,\theta_2\in \cc_-^*\cup\cc_+^*$ and assume that $\Gamma$ does not contain a face normal to $\theta_1$ nor to $\theta_2$, so that $\Gamma,\theta_i$
satisfy the hypothesis of Corollary \ref{cor:2}. We will first treat the case $\theta_1\in\cc_+^*$ and $\theta_2\in\cc_-^*$
and $\theta_1<\theta_2\in \cc_+^*$. The case $\theta_1\not=\theta_2\in\cc_-^*$ is obtained by a symmetry with respect to the horizontal axis.

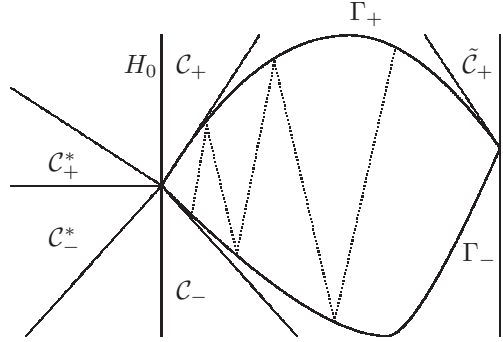
\begin{figure}[!ht]
\begin{center}
\setlength{\unitlength}{1cm}
\begin{picture}(7,5)
   \qbezier(2,2)(3.2,4)(4.5,4)
\qbezier(2,2)(4,0)(5,0)
   \qbezier(4.5,4)(5.4,4)(6.5,2.5)
\qbezier(5,0)(5.4,0)(6.5,2.5)

\dottedline{0.02}(0,2)(2,2)
\dottedline{0.02}(2,0)(2,4)
\put(1.5,3.5){$H_0$}
\dottedline{0.02}(2,2)(3.3,4)
\put(2.2,3.5){$\cc_+$}
\dottedline{0.02}(0,3.3)(2,2)
\put(0.5,2.2){$\cc_+^*$}

\dottedline{0.02}(2,2)(3.8,0)
\put(2.2,0.5){$\cc_-$}
\dottedline{0.02}(2,2)(0.2,0)
\put(0.5,1.2){$\cc_-^*$}

\dottedline{0.02}(6.5,0)(6.5,4)
\dottedline{0.02}(6.5,2.5)(5.5,4)
\put(6,3.5){$\tilde\cc_+$}

\put(4.5,4.2){$\Gamma_+$}
\put(6,1){$\Gamma_-$}
    \dottedline{0.05}(5.1,3.8)(4.3,0.2)
    \dottedline{0.05}(4.3,0.2)(3.5,3.65)
    \dottedline{0.05}(3.5,3.65)(3,1.1)
    \dottedline{0.05}(3,1.1)(2.6,2.8)
    \dottedline{0.05}(2.6,2.8)(2.4,1.6)
    \end{picture}
\caption{Closed curve with a corner point}
\end{center}
\end{figure}

Now, there is a
unique $s\in(0,1)$ that we denote by $s_*$ such that the line through $\gamma(s)$ directed by $\theta_1^\perp$ is a supporting line for $\Gamma$.
Define $\Gamma_+=\{\gamma(s),s\in(0,s_*)\}$ and $\Gamma_-=\{\gamma(s),s\in(s_*,1)\}$.
Observe that, due to the convexity of $\Gamma$, every line issued from a point $A\in\Gamma_+$ directed by $\theta_1^\perp$
will intersect $\Gamma$ again in a point $B\in\Gamma_-$. Further, a line through $B$ directed by $\theta_2^\perp$
will then intersect $\Gamma$ again in a point $C\in\Gamma_+$. The assumption on the angles imply that we go from $A$ to $C$ along $\Gamma$
clockwise. In the language of Section \ref{sec:general}, the mapping $\Phi$ is strictly decreasing on $(0,s_*)$. But then, for every
$s\in(0,s_*)$ the interval $[\Phi(s),s]$ is wandering for $\Phi$.
According to Lemma \ref{lem:wandering}, if $\mu\in\aa\cc(\Gamma)$ is such that $\widehat{\mu}=0$
on $\ell_{\theta_1}\cup\ell_{\theta_2}$, then $f_\mu=0$ on $[\Phi(s),s]$. As $s$ is arbitrary in $(0,s_*)$, $f_\mu=0$
on $(0,s_*)$. Using the fact that $\Phi$ is strictly increasing on $(s_*,1)$ we obtain that $f_\mu=0$ on $(s_*,1)$ as well.

If $\theta_1,\theta_2$ are both in $\cc_+$, a slight adaptation is needed. Without loss of generality, assume that $\theta_1<\theta_2$.
Then the same geometric argument shows that the map $\Phi$ is still strictly
decreasing and again $f_\mu=0$ on $(0,s_*)$. Let $s_*$ be defined $\bigl(\gamma(1/2)+\R\theta_-\bigr)\cap\Gamma=\{\gamma(1/2),\gamma(s_*)\}$
(note that we might have $s_*=1/2$). Corollary \ref{cor:2}-\eqref{eq:symf} for $\theta_-$ shows that $f_\mu=0$ on $(s_*,1)$.

We have thus proved the following:

\begin{proposition}
\label{prop:poly}
With the above notation, if $\theta_1\in\cc_+^*$ and $\theta_2\in\cc_-^*$. 
Then $(\Gamma,\ell_{\theta_1}\cup\ell_{\theta_2})$ is a Heisenberg Uniqueness Pair.
\end{proposition}

\begin{remark}
In the case $\gamma(1/2)$ is also a corner point, the previous proof may easily be extended to prove the following: 
Let again $\cc_\pm$ be the previous cones and define $\tilde \cc_\pm$ to be the analogous cones for $\gamma(1/2)$,
translated to have there summit at the origin. $\tilde\cc_+$ is in the upper left quadrant and $\tilde\cc_-$
in the lower left one.

Then, if $\theta_1,\theta_2$ are in $\cc_+\cup\cc_-\cup\tilde\cc_+\cup\tilde\cc_-$, 
$(\Gamma,\ell_{\theta_1}\cup\ell_{\theta_2})$ is a Heisenberg Uniqueness Pair.
\end{remark}

\begin{example} If $\Gamma$ is a convex polygon, then for almost every $\theta_1$,
there is an open interval of $\theta_2$'s such that $(\Gamma,\ell_{\theta_1}\cup\ell_{\theta_2})$ is a Heisenberg Uniqueness Pair.
In the case of a regular $n$-gon, this interval has length $\pi/n$.
\end{example}

%
%

\section{Heisenberg Uniqueness Pairs and rotation numbers}

\subsection{The ellipse revisited}
\label{sec:ellipse}

Let $\Gamma$ be an ellipse. According to the invariance properties (Inv1)-(Inv2) there is no loss of generality in assuming that
$\Gamma$ is the circle centered at $0$ of radius $1$, $\Gamma=\{\gamma(t)=(\cos 2\pi t,\sin 2\pi t),t\in [0,1)\}$.\footnote{This parametrization has been chosen
to be coherent with the usual definition of rotation numbers in the next section.}
Let $\theta_1\not=\theta_2$ be two angles. Without loss of generality $\theta_1=0$ and $\theta_2\in[0,\pi)$.
Let $\Phi_1,\Phi_2$ be the maps associated to them as in Section \ref{sec:general}.
It is easy to see that $\Phi_j$ is the orthogonal symmetry with respect to the line through $0$ directed by $\theta_j$,
in particular $\Phi_1(s)=-s\,\mbox{mod}\,1$ while $\Phi_2(s)=\frac{\theta_2}{\pi}-s\,\mbox{mod}\,1$.
Throughout the remaining of Section \ref{sec:ellipse} all functions on $\Gamma$ will be lifted as $1$-periodic functions on $\R$.

Then, according to Corollary \ref{cor:2}, $\mu\in\aa\cc(\Gamma)$ is such that $\widehat{\mu}=0$ on $\ell_0\cup\ell_{\theta_2}$
if and only if $f_\mu(-s)=-f_\mu(s)$ and $f_\mu(\theta_2/\pi-s)=-f_\mu(s)$. In particular, $f_\mu(\theta_2/\pi+s)=f_\mu(s)$.
Note that conversely, if $f_\mu(\theta_2/\pi+s)=f_\mu(s)$ and $f_\mu(-s)=-f_\mu(s)$, then $f_\mu(\theta_2/\pi-s)=-f_\mu(s)$.

According to Lemma \ref{lem:rot}, $(\Gamma,\ell_{\theta_1}\cup\ell_{\theta_2})$ is a Heisenberg Uniqueness Pair
if $\theta_2\notin\Q$. Assume now that $\dst\frac{\theta_2}{\pi}=\frac{p}{q}$, $p,q\in\N$, $p,q$ coprime.
Then every integer $j$ may be written in the form $j=kp+\ell q$ so that, if $f_\mu$ is both
$1$-periodic and $p/q$-periodic, then $f_\mu(s+j/q)=f_\mu(s+kp/q+\ell)=f_\mu(s+kp/q)=f_\mu(s)$ {\it i.e.}
$f_\mu$ is also $1/q$-periodic. The converse is trivial. Thus $\widehat{\mu}=0$ on $\ell_0\cup\ell_{\theta_2}$
if and only if $f_\mu$ is both odd and $1/q$-periodic. Such functions are all constructed in the
following way: take any $f_\mu$ on $(0,1/2q)$, extend it into an odd function on $(-1/2q,1/2q)$
and then to a $1/q$-periodic function on $\R$ (thus also to a $1$-periodic function).

This gives a more geometric and constructive proof of the following result:

\begin{theorem}[Lev \cite{Le} and Sjolin \cite{Sjo1}]
\label{th:circle}
Let $\Gamma$ be a circle and let $\theta_1,\theta_2\in\R$
Then $(\Gamma,\ell_{\theta_1}\cup\ell_{\theta_2})$ is a Heisenberg Uniqueness Pair if and only if $\dst\frac{1}{\pi}(\theta_2-\theta_1)\notin\Q$ .
\end{theorem}

For a general ellipse the condition is a bit more complicated. First let $a$ and $b$ the major and minor semi-axes of the ellipse,
so that, if we denote by $L(x,y)=(x,ay/b)$ then there is a rotation $R_\theta$ such that $LR_\theta\Gamma$
is a circle $\cc$ of radius $a$. According to the invariance properties, $(\Gamma,\ell_{\theta_1}\cup\ell_{\theta_2})$
is a Heisenberg uniqueness pair if and only if $(\cc,\ell_{\ffi_1}\cup\ell_{\ffi_2})$
with $\ell_{\ffi_j}=(R_\theta^{-1})^*(L^{-1})^*\ell_{\theta_j}$. It follows that $\ffi_2-\ffi_1=\arcsin\frac{b\sin\theta_2}{\sqrt{a^2+b^2}}-\arcsin\frac{b\sin\theta_1}{\sqrt{a^2+b^2}}$.

\begin{figure}[!ht]
\begin{center}
\setlength{\unitlength}{1cm}
\begin{picture}(6,3)
  \qbezier(5.8,2.0)(5.8,2.3728)(4.9799,2.6364)
  \qbezier(4.9799,2.6364)(4.1598,2.9)(3.0,2.9)
  \qbezier(3.0,2.9)(1.8402,2.9)(1.0201,2.6364)
  \qbezier(1.0201,2.6364)(0.2,2.3728)(0.2,2.0)
  \qbezier(0.2,2.0)(0.2,1.6272)(1.0201,1.3636)
  \qbezier(1.0201,1.3636)(1.8402,1.1)(3.0,1.1)
  \qbezier(3.0,1.1)(4.1598,1.1)(4.9799,1.3636)
  \qbezier(4.9799,1.3636)(5.8,1.6272)(5.8,2.0) 
  \dottedline{0.05}(4.9799,1.3636)(3.0,2.9)
    \dottedline{0.05}(1.0201,1.36)(3.0,2.9)
    \dottedline{0.05}(1.0201,1.36)(0.25,2.1)
    \dottedline{0.05}(0.65,2.5)(0.25,2.1)
    \dottedline{0.05}(0.65,2.5)(2.1,1.2)
    \dottedline{0.05}(4.1,2.8)(2.1,1.2)
\end{picture}
\caption{The ellipse}
\end{center}
\end{figure}
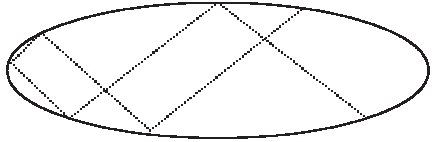

\subsection{An extension}

Let $-\pi/2<\theta_1<0<\theta_2<\pi/2$ be two angles and let $\ell=-2\tan\theta_1+2\tan\theta_2$.
Define $\Gamma=\{\gamma(t),t\in[0,1]\}$ as follows:
\begin{equation}
\label{eq:tube}
\gamma(t)=\begin{cases}
\bigl(\ell+\cos 4\pi(t-1/8) ,\sin4\pi(t-1/8)\bigr)&\mbox{for }t\in[0,1/4]\\
(\ell(2-4t),1)&\mbox{for }t\in[1/4,1/2]\\
\bigl(\cos 4\pi(t-3/8) ,\sin4\pi(t-3/8)\bigr)&\mbox{for }t\in[1/2,3/4]\\
(\ell(-3+4t),-1)&\mbox{for }t\in[3/4,1]
\end{cases}
\end{equation}
and write $\Gamma=\cc_+ \cup\ss_+\cup\cc_-\cup\ss_-$ for the four corresponding pieces of $\Gamma$
({\it see} Figure \ref{fig:tube}).

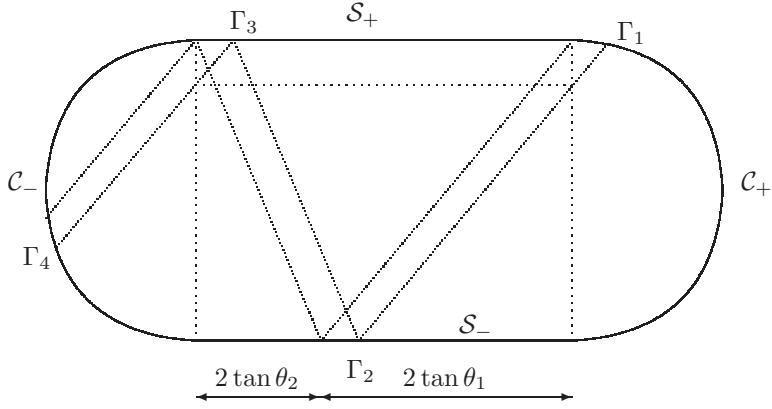
\begin{figure}[!ht]
\begin{center}
\setlength{\unitlength}{0.5cm}
\begin{picture}(26,13)
\qbezier[400](2,6)(2.2,9.8)(6,10)
\qbezier[400](2,6)(2.2,2.2)(6,2)
\put(1,6){$\mathcal{C}_-$}
\put(20.5,6){$\mathcal{C}_+$}
\put(10,10.5){$\mathcal{S}_+$}
\put(13,2.2){$\mathcal{S}_-$}
\dottedline{0.01}(6,10)(16,10)
\dottedline{0.01}(6,2)(16,2)
\qbezier[400](20,6)(19.8,9.8)(16,10)
\qbezier[400](20,6)(19.8,2.2)(16,2)
\dottedline{0.2}(6,2)(6,10)
\dottedline{0.2}(16,2)(16,10)
\dottedline{0.2}(6,8.8)(16,8.8)
\dottedline{0.1}(2.3,4.5)(7,10)
\dottedline{0.1}(10.32,2)(16.94,9.9)
\put(17.2,10){$\Gamma_1$}
\dottedline{0.1}(10.32,2)(7,10)
\put(10,1){$\Gamma_2$}
\dottedline{0.1}(2,5.27)(6,10)
\put(6.9,10.3){$\Gamma_3$}
\put(1.4,4){$\Gamma_4$}
\dottedline{0.1}(9.32,2)(15.94,9.9)
\dottedline{0.1}(9.32,2)(6,10)
\put(6,0.5){\vector(1,0){3.3}}
\put(9.3,0.5){\vector(-1,0){3.3}}
\put(6.5,0.8){$2\tan\theta_2$}
\put(9.3,0.5){\vector(1,0){6.7}}
\put(16,0.5){\vector(-1,0){6.7}}
\put(11.5,0.8){$2\tan\theta_1$}
\end{picture}
\caption{The domain $\Gamma$}
\label{fig:tube}
\end{center}
\end{figure}

In other words, $\Gamma$ is a circle of radius $1$, cut into two halves, the two halves are then separated by a distance $\ell$ and glued together
by a rectangle of length $\ell$ and width $2$. This length is chosen so that the following holds

-- take a point $\Gamma_1$ in $\cc_+$, and draw a line orthogonal to $\theta_1$ and assume this line
intersects $\ss_-$ in a point $\Gamma_2$. (Otherwise it intersects $\cc_+$ in a point $\tilde\Gamma_1$)

-- From $\Gamma_2$, draw a line orthogonal to $\theta_2$. This line will intersect $\ss_+$ in a point $\Gamma_3$.

-- From $\Gamma_3$, draw a line orthogonal to $\theta_1$. This line will intersect $\cc_-$ in a point $\Gamma_4$.

Then $\Gamma_4$ is the translate by $(-\ell,0)$ of the point $\tilde\Gamma_4$ that is the intersection of $\cc_-+(\ell,0)$
with the line orthogonal to $\theta_1$ starting at $\Gamma_1$.

We may of course exchange $\cc_+$ and $\cc_-$ that is, to go backwards in the above argument. Moreover, we can replace $\theta_1$
by $\theta_2$ (this is needed if, at the first step, we go from $\Gamma_1$ to $\tilde\Gamma_1$).

Define $s_j$, $j=1,\ldots,4$ so that $\gamma(s_j)=\Gamma_j$.

Now let $\mu\in\aa\cc(\Gamma)$ be such that $\widehat{\mu}=0$ on $\ell_{\theta_1}\cup\ell_{\theta_2}$. Then
\eqref{eq:symf} for $\theta_1$ in \eqref{eq:tub1}, then for $\theta_2$ in \eqref{eq:tub2} and for $\theta_1$ again in \eqref{eq:tub3}
shows that:
\begin{eqnarray}
\frac{f_\mu(s_1)}{\pi_{\theta_1}\gamma'(s_1)}
&=&-\frac{f_\mu(s_2)}{\pi_{\theta_1}\gamma'(s_2)}
=-\frac{f_\mu(s_2)}{\pi_{\theta_2}(4\ell,0)}\frac{\pi_{\theta_2}(4\ell,0)}{\pi_{\theta_1}(4\ell,0)}\label{eq:tub1}\\
&=&\frac{f_\mu(s_3)}{\pi_{\theta_2}(-4\ell,0)}\frac{\pi_{\theta_2}(4\ell,0)}{\pi_{\theta_1}(4\ell,0)}
=\frac{f_\mu(s_3)}{\pi_{\theta_1}(-4\ell,0)}\label{eq:tub2}\\
&=&\frac{f_\mu(s_4)}{\pi_{\theta_1}\gamma'(s_4)}.\label{eq:tub3}
\end{eqnarray}

A similar identity holds if we replace $\theta_1$ by $\theta_2$.

Let us now define $\nu$ a measure on the unit circle $\{(\cos 2\pi t,\sin2\pi t),t\in[-1/4,3/4]\}$
by 
$$
f_\nu(t)=\begin{cases}f_\mu(t/2+1/8)&\mbox{for }t\in[-1/4,1/4]\\
f_\mu(t/2+3/8)&\mbox{for }t\in[1/4,3/4]
\end{cases}.
$$
In other words, $\nu$ is $\mu$ restricted to the two half-circles (when glued back together).

From the discussion above, we see that \eqref{eq:tub3} is \eqref{eq:symf} for $f_\nu$ and $\theta_1$.
The same holds for $\theta_2$. Therefore, $\widehat{\nu}=0$ on $\ell_{\theta_1}\cup\ell_{\theta_2}$.
But, according to Theorem \ref{th:circle}, $\nu=0$, that is
$f_\mu=0$ on $[0,1/4]\cup[1/2,3/4]$. It follows from \eqref{eq:tub1}-\eqref{eq:tub2} that $f_\mu=0$
on $[1/4,1/2]\cup[3/4,1]$.

We have thus proved:

\begin{proposition}
Let $-\pi/2<\theta_1<0<\theta_2<\pi/2$ be two angles and let $\Gamma=\{\gamma(t),t\in[0,1]\}$
with $\gamma$ defined in \eqref{eq:tube}. Then
$(\Gamma,\ell_{\theta_1}\cup\ell_{\theta_2})$ is a Heisenberg Uniqueness Pair.
\end{proposition}

\subsection{Rotation numbers}
\label{sec:rot}

Till the end of section \ref{sec:rot}, we will assume that $\Gamma$ is a $\cc^k$-smooth, $k\geq 4$
closed curve with non vanishing curvature. We parametrize $\Gamma=\{\gamma(s),s\in\R\}$
where $\gamma$ is one-to-one and $1$-periodic. Let $\theta_1\not=\theta_2$ be two angles and assume that
$\Gamma$ satisfies the hypothesis of Corollary \ref{cor:2} for both $\theta_1$ and $\theta_2$.
Let $\Phi_j=\Phi_{\theta_j}$ be the corresponding maps and write 
$\Phi$
for $\Phi=\Phi_2\circ\Phi_1$.

Note that $\Phi$ is of class $\cc^{k-1}$ and, as $\Phi_1$ and $\Phi_2$ are orientation reverting,
$\Phi$ is orientation preserving. We denote by $\tilde\Phi$ a $\cc^{k-1}$ lifting of $\Phi$ as a map from $\R\to\R$.

We need a bit more notation. All results mentioned in this section are standard facts in the theory of dynamical systems
and can be found in \cite{He,Yo} which also give precise references for them.

The rotation number of $\tilde\Phi$ is defined as $\rho(\tilde\Phi)=\lim\frac{\tilde\Phi^n(x)-x}{n}$. As is well known, this limit exists and does not depend on 
$x$. Moreover, we define $\rho(\Phi)=\rho(\tilde\Phi)\,\mbox{mod}\,1$ and this number does not depend on the choice of lifting $\tilde\Phi$. 

\begin{notation}
We will write $\rho(\Gamma;\theta_1,\theta_2)=\rho(\Phi)$ to stress the dependence on $\theta_1,\theta_2$
and $\Gamma$.
\end{notation}

Recall that $\rho(\Phi)$ is rational if and only if $\Phi$ has a periodic orbit.
On the other hand if $\alpha=\rho(\Phi)$ is irrational, it is known that $\Phi$ is conjugated to the rotation of angle $\alpha$
(for this we only need $\Phi$ to be of class $\cc^2$ but $\cc^1$ may not suffice).
However this conjugation may not be regular, even though $\Phi$ is of class $\cc^\infty$. In order to obtain a regular map, we need more.
Recall that $\alpha\in\R\setminus\Q$ is called diophantian of order $\beta$ (in $\alpha\in\cc_\beta$)
if there exists $C>0$ such that $|\alpha-p/q|\geq C/q^{2+\beta}$ for every $p/q\in\Q$. Note that $\dst\bigcup_{\beta\geq0}\cc_\beta$
has full Lebesgue measure. We will use the following theorem:

\begin{theorem}[Yoccoz \cite{Yo}]\label{th:yoccoz}
If $\Phi$ is of class $\cc^{k-1}$, $k\geq 4$, and assume that $\alpha:=\rho(\Phi)\in\cc_\beta$
with $k>2(\beta+1)$. Then there exists a diffeomorphism $h$ of class $\cc^{k-\beta-2-\eps}$ for every $\eps>0$ such that
$\Phi=h^{-1}\circ R_\alpha\circ h$ where $R_\alpha$ is the rotation of angle $\alpha$, $R_\alpha(t)=t+\alpha\ \mbox{mod}\,1$.
\end{theorem}

Together with Lemma \ref{lem:rot} we obtain the following

\begin{corollary}
\label{cor:yoc}
Let $\beta\geq 0$, $k\geq \min(4,\beta+3,2\beta+2)$.
Let $\Gamma$ be a $\cc^k$ smooth closed convex curve with non-vanishing curvature and $\theta_1,\theta_2$ be two angles.
Assume that $\rho(\Gamma;\theta_1,\theta_2)\in \cc_\beta$ then $(\Gamma,\ell_{\theta_1}\cup\ell_{\theta_2})$
is a Heisenberg Uniqueness Pair.
\end{corollary}

Unfortunately, computing the rotation number $\rho(\Gamma;\theta_1,\theta_2)$ is practically impossible.
Nevertheless, if we assume that
\begin{equation}
\label{eq:tphi}
a\leq\tilde\Phi(x)-x\leq b
\end{equation}
{\it i.e.} if we bound the ``displacement'' of $\Phi$ then
$$
ka\leq \tilde\Phi^k(x)-x=\sum_{j=1}^{k}\tilde\Phi^j(x)-\tilde\Phi^{j-1}(x)\leq kb
$$
thus $a\leq \rho(\Gamma;\theta_1,\theta_2)\leq b$. Note that it is enough to obtain the bound \eqref{eq:tphi} for $x\in[0,1]$.
and that \eqref{eq:tphi} is equivalent to 
$$
a\leq|\Phi_1(x)-\Phi_2(x)|\leq b.
$$
Now, from this, it is obvious that $\rho(\Gamma;\theta_1,\theta_2)\to 0$ when $\theta_2\to\theta_1$.
On the other hand, $\rho(\Gamma;\theta_1,\theta_2)\not=0$ since $\Phi_1(x)\not=\Phi_2(x)$
(otherwise $\theta_1\not=\theta_2$ would both be normal to $\Gamma$) thus $\min_{[0,1]}|\Phi_1(x)-\Phi_2(x)|>0$
by continuity of $\Phi_1,\Phi_2$.

To overcome this and show that Heisenberg Uniqueness Pairs are frequent, we will appeal to the following

\begin{theorem}[Herman \cite{He2}]
Let $\Psi_t$ be a family of diffeomorphisms of $[0,1)$ of class $\cc^3$ such that the dependence in the parameter $t$
is of class $\cc^1$. Then either the rotation number $\rho(\Phi_t)$ does not depend on $t$ or 
there exists a set $E$ of positive Lebesgue measure such that, for every $t\in E$, $\Phi_t$ is conjugated to
a rotation with irrational angle.
\end{theorem}

\begin{corollary}
Let $k\geq4$ and let $\Gamma$ be a $\cc^k$ smooth closed convex curve with non-vanishing curvature. Then 
there exists a
set of positive Lebesgue measure $E\subset(-\pi/2,\pi/2)^2$ such that, for almost every
$(\theta_1,\theta_2)\in E$, $(\Gamma,\ell_{\theta_1}\cup\ell_{\theta_2})$
is a Heisenberg Uniqueness Pair.
\end{corollary}

%
%
%
%

\subsection{Rational rotation number is compatible with Heisenberg Uniqueness}

Let us conclude with an example of a smooth curve $\Gamma$ and angles $\theta_1,\theta_2$ such that $\rho(\Phi)$
is rational but such that $(\Gamma,\ell_{\theta_1}\cup\ell_{\theta_2})$ \emph{is} a Heisenberg uniqueness pair.

First, let $\chi$ be a $1$-periodic $\cc^\infty$ function on $\R$ such that $\supp\chi=[0,1/4]+\Z$, $0<\chi(s)<1/2$ on $(0,1/8)$
and $\chi(s)<0$ on $(1/8,1/4)$. Let $\gamma(s)=\bigl(1+\chi(s)\bigr)(\cos 2\pi s,\sin 2\pi s)$. Note that
$\gamma\,:[0,1]\to\R^2$ has the following properties
\begin{enumerate}
\item $\gamma$ is $\cc^\infty$-smooth,

\item $\gamma(s)=(\cos 2\pi s,\sin 2\pi s)$ for $s\in[1/4,1]$ {\it i.e.} $\Gamma$ contains $3/4$ of the circle $\cc$
centered at $0$ and radius $1$.

\item $|\gamma(s)|<1$ for $s\in(0,1/8)$ and $|\gamma(s)|>1$ for $s\in(1/8,1/4)$. In other words the part of $\Gamma$
in the first quadrant is inside the disc below the diagonal and outside the disc above the diagonal.

\item If $\bigl(1+\chi(s)\bigr)^2+\chi'(s)^2-2\bigl(1+\chi(s)\bigr)\chi''(s)\geq 0$ so that $\Gamma$ is convex.
\end{enumerate}

Let $\theta_1=0$ and $\dst\theta_2=\frac{\pi}{2}$ and consider the associated maps $\Phi_1,\Phi_2$ and $\Phi$
as in Section \ref{sec:general}. Note that $k/8$, $k\in\Z$ are $2$-periodic points of $\Phi$ thus $\Phi$ has rotation number $1/2$.
Then, for $s\in[1/4,3/4]$, $\Phi(s)=s+1/2$. For $s\in (3/4,7/8)$, $s-1/2<\Phi(s)<3/8$ and
for $s\in(7/8,1)$, $\Phi(s)<s-1/2$. As a consequence, if $a\in(1/4,3/8)$ $\Phi^{2k}(a)$ is increasing and bounded, therefore it converges.
The limit is a fixed point of $\Phi$ and the only possible one is $3/8$. Similarly, if $b\in(3/8,1/2)$, $\Phi^{2k}(b)$ is decreasing and bounded
and converges to $3/8$ as well. It follows that $(a,b)$ is attractive for $\Phi$.

According to Lemma \ref{lem:attract}, if $\mu\in\aa\cc(\Gamma)$ is such that $\widehat{\mu}=0$ on $\ell_0\cup\ell_{\pi/2}$,
then $\mbox{supp}\,f_\mu\cap (a,b)\subset \bigcap_{n\geq 0}[\Phi^{2k}(a),\Phi^{2k}(b)]=\{3/8\}$. As $a$
is arbitrary in $(1/4,3/8)$ and $b$ is arbitrary in $(3/8,1/2)$, $f_\mu=0$ on $(1/4,1/2)$. Using Corollary \ref{cor:2}-\eqref{eq:symf}
for $\Phi_1$, $\Phi_2$, we deduce that $f_\mu=0$ on $(0,3/2)$ and using it again for $\Phi_1$ or $\Phi_2$ we deduce that
$f_\mu=0$.

We have thus proved the following:

\begin{proposition}
\label{prop:rotrat}
There exists a smooth closed curve and two angles $\theta_1,\theta_2$ such that
$\rho(\Gamma;\theta_1,\theta_2)$ is rational and $(\Gamma,\ell_{\theta_1}\cup\ell_{\theta_2})$
is a Heisenberg Uniqueness Pair.
\end{proposition}

\section*{Acknowledgements}
The authors wish to thank Freddy Manning and Nicolas Gourmelon for helpful discussions.

The authors kindly acknowledge financial support from the French ANR programs ANR
2011 BS01 007 01 (GeMeCod), ANR-12-BS01-0001 (Aventures).
This study has been carried out with financial support from the French State, managed
by the French National Research Agency (ANR) in the frame of the ”Investments for
the future” Programme IdEx Bordeaux - CPU (ANR-10-IDEX-03-02).

Part of this work was conducted while the first author was
visiting the Erwin Shr\"odinger Institute, Vienna, Austria.

\end{document}